\documentclass[12pt]{article}

\usepackage[margin=1.3in,marginparwidth=0.8in, marginparsep=0.1in]{geometry}

\pdfoutput=1

\usepackage{soul}

\usepackage{amssymb,latexsym,amsmath,graphicx, bbm,tikz-cd}
\usepackage{stmaryrd} 
\usepackage{amscd}
\usepackage{cite}
\usepackage{todonotes}
\usepackage{color}
\usepackage[all,cmtip]{xy}

\usepackage{amsfonts} %%%% the usual
\usepackage{mathpazo}

\usepackage[bookmarks=true, bookmarksopen=true,%
bookmarksdepth=3,bookmarksopenlevel=2,%
colorlinks=true,%
linkcolor=blue,%
citecolor=blue,%
filecolor=blue,%
menucolor=blue,%
urlcolor=blue]{hyperref}

\usepackage{times}
\usepackage{mathrsfs}
\usepackage{mathtools}
\usepackage{parskip}

\newtheorem{lemma}{Lemma}[section]
\newtheorem{proposition}[lemma]{Proposition}
\newtheorem{corollary}[lemma]{Corollary}
\newtheorem{theorem}[lemma]{Theorem}

\newtheorem{definition}[lemma]{Definition}

\newtheorem{remark}{Remark}
\newtheorem{example}{Example}

\newenvironment{proof}{{\bf
		Proof.}}{$\blacksquare$ \vspace{2mm}}

\newcommand{\B}{\mathbb{B}}
\newcommand{\C}{\mathbb{C}}

\newcommand{\G}{\mathbb{G}}
\renewcommand{\H}{\operatorname{H}}

\newcommand{\N}{\mathbb{N}}

\renewcommand{\P}{\mathbb{P}}
\newcommand{\Q}{\mathbb{Q}}
\newcommand{\R}{\mathbb{R}}

\newcommand{\Z}{\mathbb{Z}}

\newcommand{\cE}{\mathcal{E}}

\newcommand{\cG}{\mathcal{G}}

\newcommand{\cL}{\mathcal{L}}

\newcommand{\cO}{\mathcal{O}}
\newcommand{\cP}{\mathcal{P}}
\newcommand{\cQ}{\mathcal{Q}}

\newcommand{\cT}{\mathcal{T}}

\newcommand{\cV}{\mathcal{V}}
\newcommand{\cW}{\mathcal{W}}

\newcommand{\fd}{\frak{d}}

\newcommand{\ft}{\frak{t}}

\newcommand{\XX}{X}
\newcommand{\TT}{T}
\newcommand{\GG}{G}

\DeclareMathOperator{\Hom}{Hom}

\DeclareMathOperator{\Ext}{Ext}
\DeclareMathOperator{\Spec}{Spec}
\DeclareMathOperator{\Proj}{Proj}

\newcommand{\rk}{\operatorname{rk}}
\newcommand{\Sym}{\operatorname{Sym}}
\newcommand{\End}{\operatorname{End}}

\newcommand{\fg}{\mathfrak{g}}
\newcommand{\gvc}{\fg^{\vee}}

\newcommand{\gr}{\eta} %real G moment
\newcommand{\gc}{\lambda} %complex G moment
\newcommand{\loops}{\mathcal{L}}
\newcommand{\tloops}{\widetilde{\loops}}

\newcommand{\fixed}{p}

\newcommand{\degree}{\gamma}

\newcommand{\qs}{\sheaf}

\newcommand{\edges}{E}
\newcommand{\vertices}{V}

\newcommand{\Maps}{\operatorname{Maps}}

\newcommand{\gaugegrp}{G}

\newcommand{\ev}{\operatorname{ev}}

\newcommand{\gauge}{g}
\newcommand{\sheaf}{\mathcal{F}}

\newcommand{\refd}{\operatorname{ref}}

\newcommand{\framevertex}{v_\framed}
\newcommand{\framed}{f}
\newcommand{\mass}{\mathcal{M}}

\newcommand{\cohweight}{\tau}

\newcommand{\loopspacetor}{\mathcal{T}}

\newcommand{\per}{\mathcal{P}}

\newcommand{\tgtquasimaps}{\frak{T}}
\newcommand{\DTgenerator}{\Upsilon}
\newcommand{\DTrefdgenerator}{\Upsilon^{\refd}}

\newcommand{\bigsimple}{L}

\newcommand{\degreebound}{N}
\newcommand{\degreelattice}{\fg_\Z}
\newcommand{\grdim}{\operatorname{grdim}}

\newcommand{\flavourlie}{\frak{t}}

\newcommand{\indexset}{\mathcal{I}}

\newcommand{\arrangement}{A}
\newcommand{\bounded}{\mathcal{B}}
\newcommand{\feasible}{\mathcal{F}}

\newcommand{\kahlerchambers}{\frak{K}}
\newcommand{\eqchambers}{\frak{E}}
\newcommand{\character}{\chi}

\newcommand{\quantumhyp}{U_\eta}
\newcommand{\qDelta}{\mathbf{\Delta}}
\newcommand{\Att}{\mathfrak{Q}}
\newcommand{\optionalchamber}{}

\newcommand{\coordset}{E}
\newcommand{\sheafQ}{\mathcal{Q}}

\newcommand{\sheafHom}{\mathcal{H}om}

\DeclareMathAlphabet\mathbfcal{OMS}{cmsy}{b}{n}
\newcommand{\qbounded}{\mathbfcal{B}}
\newcommand{\qfeasible}{\mathbfcal{F}}
\newcommand{\qarrangement}{\mathbf{\arrangement}}
\newcommand{\qeta}{\eta}

\newcommand{\smallproj}{P}
\newcommand{\tilting}{T^!}

\newcommand{\twist}{\mathbf{m}}
\newcommand{\twistspace}{\mathbb{M}}
\newcommand{\loopcochar}{\tilde{\zeta}}

\newcommand{\QM}{\frak{Q}}
\newcommand{\rotations}{\C^\times_q}
\newcommand{\chain}{\delta}

\newcommand{\smon}{\mathbb{S}}
\newcommand{\bases}{\mathbb{B}}
\newcommand{\boufeas}{\cP}
\newcommand{\tiltingnoshriek}{T}
\newcommand{\loopnumber}{n}

\DeclareMathOperator{\codim}{codim}

\DeclareMathOperator{\sign}{sign}

\title{Twisted Quasimaps and Symplectic Duality for Hypertoric Spaces}

	\author{Michael McBreen$^{[1]}$, Artan Sheshmani$^{[2,3,4]}$, Shing-Tung Yau$^{[5]}$}

	\title{Twisted Quasimaps and Symplectic Duality for Hypertoric Spaces}
	
\begin{document}
	\maketitle
\smallskip

\begin{center}
\textit{To the memory of Thomas Andrew Nevins (June 14, 1971--February 01, 2020)}
\end{center}
\begin{abstract}We study moduli spaces of twisted quasimaps to a hypertoric variety $\XX$, arising as the Higgs branch of an abelian supersymmetric gauge theory in three dimensions. These parametrise general quiver representations whose building blocks are maps between rank one sheaves on $\P^1$, subject to a stability condition, associated to the quiver, involving both the sheaves and the maps. We show that the singular cohomology of these moduli spaces is naturally identified with the Ext group of a pair of holonomic modules over the `quantized loop space' of $\XX$, which we view as a Higgs branch for a related theory with infinitely many matter fields. We construct the coulomb branch of this theory, and find that it is a periodic analogue of the coulomb branch associated to $\XX$. Using the formalism of symplectic duality, we derive an expression for the generating function of twisted quasimap invariants in terms of the character of a certain tilting module on the periodic coulomb branch. We give a closed formula for this generating function when $\XX$ arises as the abelianisation of the $N$-step flag quiver.
\smallskip


\end{abstract}
	
\tableofcontents

		\section{Introduction}

		A large body of geometric representation theory in the last decade has grown around the study of symplectic resolutions: algebraic symplectic varieties which are `almost affine' in a suitable sense. Their quantizations yield important algebras in representation theory, such as the envelopping algebras of reductive lie algebras, rational Cherednik algebras \cite{etingof2002symplectic} and finite W-algebras \cite{premet2002special}. Their enumerative geometry, on the other hand, has been related to quantum integrable systems attached to quantum loop groups; see for instance \cite{MR3951025, MR3881458, pushkar2020baxter}. 
		
		Relations between the quantization of a symplectic resolution in finite characteristic and its enumerative geometry have been conjectured by Bezrukavnikov and Okounkov \cite{bezrukavnikovmonodromy}, and in certain cases proved \cite{anno2015stability}. A second line of investigation relates quantizations in characteristic zero of a symplectic resolution $\XX$ to the enumerative geometry of a `symplectic dual' resolution $\XX^!$; this paper takes a further step in that direction. 
		
		The work \cite{MR3594665} defined an analogue of the BGG category $\cO$ for symplectic resolutions, and conjectured that these occur in pairs $\XX, \XX^!$ such that category $\cO$ of $\XX$ is Koszul dual to category $\cO$ of $\XX^!$. They called $\XX^!$ the symplectic dual of $\XX$. 
		Gukov and Witten pointed out that the exact same pairs $\XX, \XX^!$ as Higgs and Coulomb branches of supersymmetric gauge theories in three dimensions. These theories should in turn come in dual pairs, exchanging their Higgs and Coulomb branches, as explained by Seiberg and Intriligator \cite{intriligator1996mirror}.  
		
		A construction of the Coulomb branch (or from our perspective, of $\XX^!$, starting from $\XX$), was proposed in the papers \cite{MR3565863, MR3952347}. A physical construction, similar in spirit, was also proposed in \cite{MR3893100}. 
		
		In \cite{hikita2017algebro}, Hikita conjectured a second, rather surprising relationship between $\XX$ and $\XX^!$: an isomorphism between the cohomology ring $\H^{\bullet}(\XX,\C)$ and the ring of coinvariants of $\cO(\XX^!)$ under the action of a torus $T$ of Hamiltonian automorphisms of $\XX^!$. This conjecture was extended to equivariant cohomology by Nakajima \cite[Conjecture 8.9]{kamnitzer2015highest}; the corresponding deformation of the coinvariant algebra is the so-called $B$-algebra of a quantization of $\XX^!$. Interestingly, the definition of the latter depends on a choice of cocharacter $\zeta$ of $T$.
		
		In \cite{MR4295090}, the authors conjectured that the quantum D-module of $\XX$, in a certain specialisation, equals the `character D-module' of $\XX^!$. The latter is defined via the quantization of $\XX^!$, and in particular describes the differential equations satisfied by the characters of modules in category $\cO$. 
		
		In this paper, we consider a specific enumerative problem on $\XX$, namely the Betti numbers of the moduli of twisted hypertoric quiver sheaves on a rational curve, which one may think of as a kind of refined Donaldson Thomas invariant. These are assembled into a generating function
		
		\[ 	\DTrefdgenerator(z,\tau) = \sum_{\degree \in \H_2(\XX, \Z)} \sum_{i \in \Z} (-1)^i \dim \H^i(\QM_\twist(\mathbb{P}^1, \XX, \degree),\C) z^{\degree} \tau^i. \]
		We define a symplectic ind-scheme $\tloops \XX$, which we view as a model of the universal cover of the loop space of $\XX$. We show that the moduli of twisted quiver sheaves may be expressed as an intersection of  lagrangians in $\tloops \XX$. We then propose an extension of symplectic duality to the infinite dimensional space $\tloops \XX$, and identify its dual $\tloops \XX^!$ with a {\em periodic analogue} $\per \XX^!$ of $\XX^!$. This space, which is finite dimensional but of infinite type, carries an action of $\H_2(\XX, \Z)$ by automorphisms. It was first defined by Hausel and Proudfoot in an unpublished note.  When the hypertoric variety is cographical, i.e. arises from a graph $\Gamma$ in a suitable sense, the space $\per \XX^!$ is closely related to the compactified Jacobian of a nodal curve with dual graph $\Gamma$. In particular, in \cite{dancso2019deletion} it was proven that the cohomology of the quotient of (a deformation retract of) $\per \XX^!$ by its $\H_2(\XX, \Z)$ may be identified with the cohomology of the compactified Jacobian.
		
		Our main result expresses the generating function for twisted DT invariants of the hypertoric space $\XX$ as a certain graded trace of an indecomposable tilting module $T^!_{\nu(\alpha_+)^\infty}$ over the quantization of $\per \XX^!$. 
		
		\begin{theorem} \label{thm:mainthm0} [Theorem \ref{thm:mainthm}]
			\begin{equation} \label{eq:mainthm0} \DTrefdgenerator(z, \tau) = \sum_{\degree \in \H_2(\XX, \Z)} \grdim \left(  e_{\partial \degree \cdot \alpha_-^\infty}\tilting_{\nu(\alpha_+)^\infty} \right) z^\degree \tau^{-d_{\degree}}. \end{equation} \end{theorem}
		
		Theorem \ref{thm:mainthm0} requires many technical preliminaries to state, but it has a simple consequence : an explicit formula for the generating function.
		
		\begin{theorem} \label{thm:explicit0}  [Theorem \ref{thm:explicit}]
			\begin{equation} \DTrefdgenerator(z, \tau) = \sum_{b \in \bases^!} \sum_{s \in \smon^b_{\alpha_+}, r \in \smon^{b}_{\alpha_-}} \tau^{\psi_b(r+s,s)} z^{\phi_b(s + r)}. \end{equation}
		\end{theorem}
		Here $\bases^!$ indexes torus fixed points of $\XX$, and the other quantities are explained in the body of the paper. This formula may of course be obtained by other, more direct means, but it appears here as a natural expression of representation theoretic structures on the Coulomb branch. We may summarize our computations by the following very schematic diagram:
		
		\begin{center}
			\begin{tikzcd}%[column sep=1.em]
				\text{ Refined quasimap invariants of } \XX_\eta \arrow[d, leftrightarrow] \\ \text{ Ext groups of simple modules over quantum } \tloops \XX_\eta  \arrow[d,leftrightarrow, "\text{ Symplectic duality }"] 
				\\  \text{ Weight spaces of a tilting module over quantum } \per \XX_\eta^! \arrow[d]  \\ 
				\text{ \small{Explicit formulae for quasimap invariants}} 
			\end{tikzcd}
		\end{center}
		
		An appealing feature of our approach is that that we can deduce Theorem \ref{thm:mainthm0} directly from a Koszul duality between category $\cO$ for $\tloops \XX$ and $\per \XX^!$. The Koszul duality, established in the finite dimensional setting in \cite{BLPWtorico}, is a basic expected feature of symplectically dual spaces. Thus we are able to relate in a precise way two seemingly distinct relationships between dual resolutions: one categorical, the other enumerative. 
		
		There is an additional wrinkle to our story. Category $\cO$ depends on two parameters : a cocharacter of the hamiltonian torus acting on $\XX$, and a stability condition determining a birational model of $\XX$. We are interested in simple modules lying in category $\cO$ on $\loops \XX$ for {\em opposite} torus cocharacters. Dually, we must consider category $\cO$ for opposite birational models of  $\per \XX^!$. This leads us to compose Koszul duality with Ringel duality, which eventually explains the appearance of the tilting module.
		
		We should note that to avoid dealing with the potential pathologies of infinite dimensional spaces, we work extensively with finite dimensional and finite type approximations to $\tloops \XX$ and $\per \XX^!$, and limits of these. It would be interesting to work directly on the limit spaces, and develop in this context the full analogues of the finite dimensional theory - module categories, Koszul dualities and their ilk. A second interesting direction is to replace the hypertoric space $\XX$ by a Nakajima quiver variety, or more generally the Higgs branch of a non-abelian reductive group $G$. The analogue of $\per \XX^!$ in this case may be a periodic version of the Coulomb branch of $\XX$ defined in \cite{MR3952347}. Our approach may also be compared to the interesting paper \cite{MR3893100}; we hope that our perspective will be complementary to that one. The reader may also compare our description of the quasimap moduli spaces with that of \cite{MR4616688}.
		
		The structure of our paper is as follows. We begin with a review of hypertoric varieties, their quantizations and the module categories attached to these, as described in \cite{musson1998invariants, BLPWtorico, GDKD}. We hope these sections will be helpful to readers less familiar with the combinatorics of hypertoric spaces. We then turn to enumerative geometry, and recall the definition of twisted quasimaps from \cite{kim2016stable}.  The last two sections of our paper introduce the hypertoric loop space and its symplectic dual, and apply the general theory from the previous sections to these rather unusual hypertoric varieties to obtain our formulae for quasimap invariants.
		\subsection{Acknowledgements}
		The first named author thanks Roman Bezrukavnikov and Andrei Negut for insights arising from old and new joint projects, Conan Leung and Du Pei for helpful comments and Ben Webster for explanations of Koszul duality for hypertoric varieties. Research of first named author was supported by a CUHK Start-up grant (project code : 4937006) and the starting grant of A. S. at Institut for Matematik, Aarhus Universitet. The second named author would like to thank Ludmil Katzarkov and Vladimir Baranovsky, for valuable conversations on geometric quantization. A. S. is further grateful to Sheldon Katz,  Alina Marian and Martijn Kool who suggested the problem of counting stable flag shaves with support on curves (analog of twisted quiver sheaves discussed in the current article), while discussing earlier work of A. S. on higher rank flag sheaf invariants on surfaces, and their relation to Vafa-Witten theory.  Research of A. S. was further partially supported by NSF DMS-1607871, NSF DMS-1306313, Simons 38558, and Laboratory of Mirror Symmetry NRU HSE, RF Government grant, ag. No 14.641.31.0001. A.S, would like to sincerely thank Center for Mathematical Sciences and Applications at Harvard University, as well as the Aarhus University department of mathematics, and the Laboratory of Mirror Symmetry in Higher School of Economics, Russian federation, for their support. S.-T. Y. was partially supported by NSF DMS-0804454, NSF PHY-1306313, and Simons 38558.
		
		\section{Symplectic resolutions and symplectic duality}
		
		We summarize the general features of symplectic duality, before passing to the hypertoric setting in the next section.
		
		\begin{definition}
			Let  $\XX$ be a smooth complex variety equipped with an algebraic symplectic form $\Omega$ and an action of $\C^\times$ scaling $\Omega$ by a nontrivial character. We call $\XX$ a conical symplectic resolution if 
			\begin{itemize}
				\item The natural map $\XX \to \Spec H^0(\XX, \mathcal{O}_\XX)$ is proper and birational.
				\item The induced $\C^{\times}$-action on $\Spec H^0(\XX, \mathcal{O}_\XX)$ contracts it to a point.
				\item The minimal symplectic leaf of $\Spec H^0(\XX, \mathcal{O}_\XX)$ is a point. 
			\end{itemize}
		\end{definition}
		The last condition is to avoid cases such as $\XX = \C^2$, and can often be removed at the cost of slightly more cumbersome statements. Famous examples include the Springer resolution $T^{\vee}G/B$, moduli of framed sheaves on $\C^2$ and Nakajima quiver varieties.
		
		We fix a maximal torus $T$ of the group of (complex) hamiltonian automorphisms of $\XX$, which we assume, for simplicity, acts with isolated fixed points on $\XX$. The ring of algebraic functions on $\XX$ can be quantized to obtain an $\N$-graded noncommutative algebra $U_\eta$ depending on a parameter $\eta \in H^2(\XX, \C)$. Given a cocharacter $\zeta$ of $T$ with isolated fixed points on $\XX$, we can decompose $U_\eta$ into subalgebras $U^+_\eta, U^-_\eta, U^0_\eta$ scaled positively, negatively or not at all by $\zeta$. 
		
		Category $\cO$ is defined as the category of finitely generated modules over $U_\eta$ on which $U_\eta^+$ acts locally finitely. 
		
		In \cite{MR3594665}, the authors define a symplectic duality between two conical symplectic resolutions $\XX$ and $\XX^!$ as
		\begin{itemize}
			\item Isomorphisms $T \cong H^2(\XX^!, \C^\times)$ and $T^! \cong H^2(\XX, \C^\times)$, identifying certain root hyperplanes defined in \cite{MR3594665}. In particular, any choice of cocharacter $\zeta$ of $T$ determines a choice of $\eta \in H^2(\XX^!, \C)$, and vice-versa. 
			\item A Koszul duality (see Definition \ref{def:koszulduality}) between category $\cO$ of $\XX$ and category $\cO$ of $\XX^!$, where the parameters $\zeta, \eta$ and $-\zeta^!, -\eta^!$ are identified by the above.
		\end{itemize}
		The original symplectic duality, from this perspective, was the Koszul duality of Category $\cO$ for a reductive Lie algebra $\frak{g}$ and its Langlands dual $\frak{g}^L$, together with its extension to parabolic and singular variants \cite{Soe90, beilinson1986mixed, BGS96}. 
		
		A physical interpretation of Koszul duality in the context of symplectic duality was given in \cite{MR3893100}, where it is explained as a correspondence of boundary conditions for supersymmetric gauge theories in three dimensions.
		
		\section{Hypertoric varieties} 
		In this section we define our main geometric actors: the hypertoric varieties introduced in \cite{bielawski2000geometry}. For a survey of these spaces, see \cite{proudfoot2006survey}. 
		
		Fix the following data:
		\begin{enumerate}
			\item A finite set $\coordset$.
			\item A short exact sequence of complex tori 
			\begin{equation} \label{eq:basicsequence} 1 \to G \to D \to T \to 1, \end{equation} with an isomorphism $D = (\C^\times)^\coordset$.
			\item A character $\eta$ of $G$.
		\end{enumerate}
		To these choices we will associate a hypertoric variety. Let $\fg, \fd, \ft$ be the complex lie algebras of $G,D,T$. We require that $\fd_{\Z} \to \ft_{\Z}$ be totally unimodular, i.e. the determinant of any square submatrix (for a given choice of integer basis) is one of $-1,0,1$. This will ensure that our hypertoric variety is a genuine variety and not an orbifold. We also assume that no cocharacter of $\GG$ fixes all but one of the coordinates of $\C^\coordset$.
		
		Let $V := \Spec \C[z_e | e \in \coordset]$; then $D$ acts by hamiltonian transformations on $T^{\vee}V = \Spec \C[z_e, w_e | e \in \coordset]$, equipped with the standard symplectic form $\Omega := \sum_{e \in \coordset} dz_e \wedge dw_e$.  A moment map $\mu_D: T^{\vee}V\to \fd^{\vee}$ is given by
		\[
		\mu_D(z,w) = (z_e w_e).
		\]
		We have the exact sequence
		\begin{equation} \label{basicsequence}
			0\to\fg\overset{\partial}{\to}\fd \to \ft\to 0
		\end{equation}
		and its dual
		\begin{equation} \label{dualsequence}
			0\to \ft^{\vee} \to \fd^{\vee}\overset{\partial^{\vee}}{\to} \fg^{\vee}\to 0.
		\end{equation}
		
		The pullback $\mu_\GG=\partial^{\vee}\circ\mu_D$ defines a moment map for the $\GG$ action on $T^{\vee}V$. Fix a character $(\gr, \gc) \in \fg^{\vee}_\Z \oplus \gvc$. 
		
		\begin{definition}
			Let
			\begin{align}
				\label{def:hypertoricreduction} X_{\gr, \gc} & := \mu_{G}^{-1}(\gc) \sslash_{\gr} \GG %\\
			\end{align}
			where for $U$ a $G$-variety, $U \sslash_{\gr} \GG$ indicates the GIT quotient $\Proj \bigoplus_{m \in \N} \{ f \in \mathcal{O}(U) : g^* f = \gr(g)^m f. \}$.
		\end{definition}
		We will henceforth always assume that $\eta$ is suitably generic, in which case $\XX_{\gr, \gc}$ is smooth; this holds away from a finite set of hyperplanes. We write $X_{\gr} := X_{\gr, 0}$, which we sometimes abbreviate further to $\XX$. The Kirwan map gives identifications
		$\H^2(\XX_\gr, \Z) \cong \frak{g}^{\vee}_\Z$ and $\H_2(\XX_\gr, \Z) \cong \fg_\Z$, and $\XX_\gr$ carries a real symplectic form of class $\gr$, for which the action of the compact subtorus of $T$ is Hamiltonian.

		$\XX$ inherits an algebraic symplectic structure from its construction via symplectic reduction. The induced $\TT$ action on $\XX$ is Hamiltonian. There is a further action of $\C^{\times}_{\hbar}$ dilating the fibers of $T^{\vee}V$, which scales the symplectic form by $\hbar$. This preserves $\mu_\GG^{-1}(0)$, and descends to an action of $\C^{\times}_{\hbar}$ on $\XX$ commuting with the action of $\TT$. 
		
		The natural map $\XX_{\eta} \to \Spec \H^0(\XX_{\eta}, \mathcal{O}_{\XX_{\eta}})$ is proper and birational, and defines a symplectic resolution.

		\subsection{Hyperplane arrangements and their bounded and feasible chambers} \label{sec:arrangements}
		In the next few subsections we introduce some notions from linear programming which capture both the geometry of $\XX$ and the behavior of modules over its quantization. This material is covered in greater generality in \cite{BLPWtorico}. 
		
		To the sequence \ref{eq:basicsequence} and the character $\eta$ we associate a `polarized hyperplane arrangement' as follows.
		
		Let $\frak{t}^{\vee}_\R \to \frak{d}^{\vee}_\R \xrightarrow{\partial^{\vee}} \frak{g}^{\vee}_\R$  be the induced short exact sequence of dual lie algebras, and let $\frak{t}^{\vee}_\eta = (\partial^{\vee})^{-1}(\eta)$. It is an orbit of $\frak{t}^{\vee}_\R$; in \cite{BLPWtorico} the corresponding object is denoted $V_\R$. 
		\begin{definition}
			Let $\arrangement_\eta$ be the affine hyperplane arrangement on $\frak{t}^{\vee}_\R$  whose hyperplanes $H_e = \{ d_e = 0 \}$ are the intersections of $\frak{t}^{\vee}_\eta$ with the coordinate hyperplanes of $\frak{d}^{\vee}_\R$. 
		\end{definition} 
		Each hyperplane is cooriented, i.e. defines a positive half-space $\{ d_e \geq 0 \}$ and a negative halfspace $\{ d_e \leq 0 \}$.
		
		Each sign vector $\alpha \in \{ +, - \}^{\edges}$ determines an intersection of halfspaces $\Delta_\alpha = \{ d_e \geq 0 | \alpha(e) = + \} \cap \{ d_e \leq 0 |  \alpha(e) = - \}$ in $\frak{t}^{\vee}_\eta$. We call such intersections chambers, and will sometimes abuse notation and call $\alpha$ itself a chamber.
		
		\begin{definition} \label{def:feasible}
			We say $\alpha$ is feasible if $\Delta_\alpha$ is non-empty. 
		\end{definition}
		We write $\feasible_\eta$ for the set of feasible sign vectors. The $\Delta_\alpha$ for $\alpha$ feasible are the chambers (in the usual sense) of $\arrangement_\eta$. 
		
		Let $\Delta_{0,\alpha}$ be defined the same way as $\Delta_\alpha$, with $\eta=0$. Fix $\zeta \in \ft_\Z$. 
		\begin{definition} \label{def:bounded}
			We say $\alpha$ is bounded if $\langle \zeta, - \rangle$ is bounded and proper on $\Delta_{0,\alpha}$.
		\end{definition}
		This notion depends on $\zeta$ but not $\eta$. We write $\bounded^\zeta$ for the set of bounded chambers, and $\boufeas_{\eta}^{\zeta}$ for the set of bounded and feasible chambers.
		
		\begin{figure}[h]
			\large
			\centering{
				\resizebox{90mm}{!}{\includegraphics{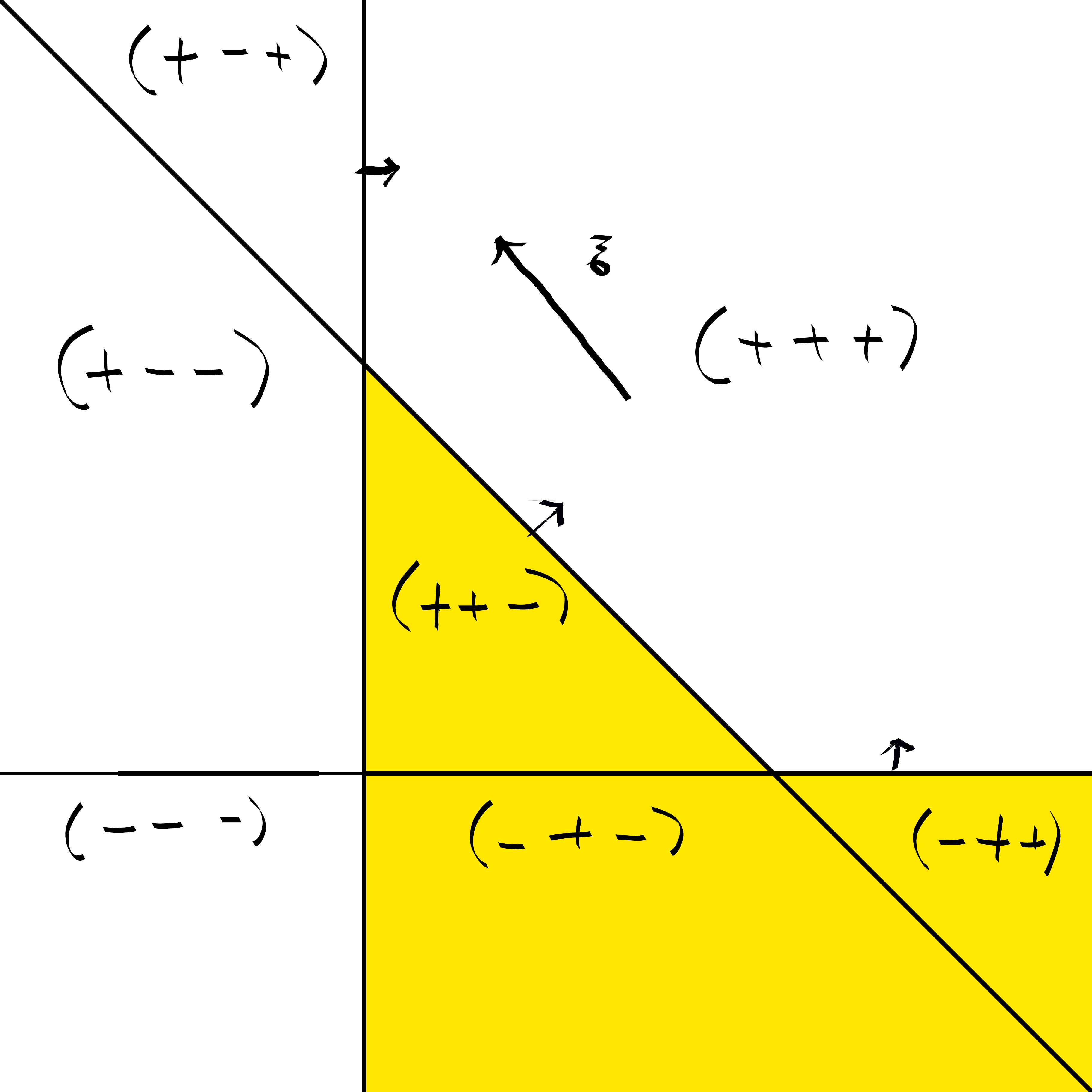}}
				\caption{A sample arrangement. We have indicated the gradient of $\langle \zeta, - \rangle$ by an arrow. The $d_e \geq 0$ halfspace for each edge $e$ is indicated by a small arrow along the $d_e=0$ hyperplane. The chambers are intersections of half planes, labeled by sign vectors $\alpha \in \{ +, - \}^3$. Bounded chambers are shaded.}
				\label{fig:chambers}
			}
		\end{figure}
		
		\subsection{Lagrangians from chambers} \label{sec:lagsfromchambers}
		To each chamber $\Delta_\alpha$, we can associate a Lagrangian $\frak{L}_{\alpha} \subset \XX_\eta$ as follows:
		\begin{equation} \label{eq:lagfromchambers} \frak{L}_{\alpha} := \{ y_e = 0 | \alpha(e) = + \} \cap \{ x_e = 0 | \alpha(e) = - \} \sslash_\eta \gaugegrp. \end{equation}
		The lagrangian $\frak{L}_{\alpha}$ is nonempty precisely when $\alpha$ is feasible. It is contracted to a point by flowing along the cocharacter $\zeta$ precisely when $\alpha$ is bounded. The chamber $\Delta_\alpha$ may be recovered as the image of $\frak{L}_{\alpha}$ under the moment map $\mu_\R : \XX_\eta \to \frak{t}^{\vee}_\eta$ with respect to the real symplectic form on $\XX_\eta$. 
		
		These lagrangians capture the geometry of $\XX_\eta$ in the following sense: 
		\begin{proposition} \cite{bielawski2000geometry}
			The union of $\frak{L}_\alpha$ over all feasible $\alpha$ is a deformation retract of $\XX_\eta$. 
		\end{proposition}
		We call this union the `core' of $\XX_\eta$.

		\subsection{Vertices and torus fixed points} 
		The vertices of our arrangement are indexed by {\em bases}, i.e. subsets $b \subset \edges$ such that $H_b := \cap_{e \in b} H_e$ is a point. Alternatively, they are the subsets indexing tuples of coordinate vectors in $\fd_\Q$ whose image in $\ft_\Q$ form a basis (our unimodularity assumption ensures that they in fact form a basis of $\ft_\Z$). This shows that the set of bases $\mathbb{B}$ does not depend on the choice of $\eta$.  
		
		\begin{definition} Let $\phi_b : \Z^{b} \to \ft^{\vee}_\Z$ be the isomorphism defined by the {\em dual} basis to the basis described above. \end{definition}

		\begin{lemma} \label{lem:chamberstovertices} 
			For generic $\zeta, \eta$, the set of feasible and bounded chambers admits a bijection \begin{equation} \label{def:mu}
				\mu : \bases \to \boufeas^{\zeta}_{\eta}
			\end{equation} 
			fixed by the condition that $\Delta_{\mu(b)}$ has $\zeta$-maximum $H_b$. 
		\end{lemma}

		Just as each chamber $\Delta_\alpha$ defines a lagrangian, each base $b \in \bases$, defines a $T$-fixed point $\fixed_b \in \XX_\eta$. 
		\begin{lemma}
			There is a bijection $\bases \to \XX_\eta^T$ taking $b$ to $$\fixed_b := \left( T^{\vee} \C^{\edges \setminus b} \cap \mu_\gaugegrp^{-1}(0) \right) \sslash_\eta \gaugegrp.$$
		\end{lemma} 
		The map $\phi_b$, in this interpretation, is given by taking linear combinations of the characters appearing in the normal bundle to $\fixed_b$ in $\XX_\eta$. On the other hand, $\frak{L}_{\mu(\alpha)}$ is the attracting cell of the fixed point $p_b$ under the action of the cocharacter $\zeta : \C^\times \to T$.

		\subsection{Equivariant and K\"ahler chambers}
		In this section, we describe the depence of $\XX_\eta$ on the parameter $\eta$, and the dependence of the fixed locus $\XX^{\C^\times}_\eta$ on the cocharacter $\zeta : \C^\times \to T$. This leads to the notion of root hyperplanes in $\frak{g}^{\vee}_\R$ and $\frak{t}_\R$. 
		
		\begin{definition} \label{circuit}
			The {\em support} of an element $\textbf{y} \in \frak{d}_\Z$ is the smallest coordinate subspace containing $\textbf{y}$. A {\em circuit} $\degree$ is a nonzero primitive element of $\fg_\Z$ whose image in $\fd_\Z$ has minimal support. A {\em root hyperplane} in $\fg_\R^{\vee}$ is a hyperplane $\degree^{\perp} \subset \fg_\R^{\vee}$ where $\degree$ is a circuit. 
		\end{definition} 
		
		\begin{proposition} 
			$\XX_{\eta}$ is smooth precisely when $\eta$ does not lie on a root hyperplane. 
		\end{proposition}
		We write $\kahlerchambers$ for the set of connected components of the central arrangement in $\frak{g}^{\vee}_\R$ defined by the root hyperplanes, which we call K\"ahler chambers. Their importance for us lies in the following fact.
		\begin{proposition} \label{prop:kahlersindexfeasibles}
			The set of feasible chambers $\feasible_\eta$ depends only on the K\"ahler chamber containing $\eta$.
		\end{proposition}
		One may view this proposition as a combinatorial manifestation of the previous one, in the sense that as $\eta$ approaches a root hyperplane, some chamber $\Delta_{\alpha}$ will collapse to a lower-dimensional polytope, and correspondingly the lagrangian $\frak{L}_\alpha \subset \XX_\eta$ will collapse to a lower-dimensional variety, thus producing a singularity of $\XX_\eta$. 
		
		There is a second central arrangement attached to $G \to D \to T$, dual in a sense we shall make precise later. 
		\begin{definition} \label{cocircuit}
			A cocircuit is a nonzero primitive element $\character$ of $\frak{t}_\Z^{\vee}$ whose image in $\frak{d}^{\vee}_\Z$ has minimal support. A {\em root hyperplane} in $\ft_\R$ is a hyperplane $\character^{\perp} \subset \ft_\R$ where $\character$ is a cocircuit. \end{definition} 
		
		%\begin{definition}
		We define the equivariant chambers of the sequence $G \to D \to T$ as the set of chambers of the central arrangement in $\frak{t}_\R$ defined by the root hyperplanes. We write $\eqchambers$ for the set of equivariant chambers. Let $\zeta : \C^{\times} \to T$ be a cocharacter, and write $\XX^{\zeta}$ for the set of fixed points under the induced $\C^{\times}$-action. The following propositions are easily verified.
		
		\begin{proposition}
			$\XX^{\zeta}$ is discrete precisely when $\zeta$ lies in an equivariant chamber. 
		\end{proposition}
		\begin{proposition} \label{prop:eqsindexboundeds}
			The set of bounded chambers $\bounded^\zeta$ depends only on the equivariant chamber containing $\zeta$.
		\end{proposition}

		%\end{definition}
		
		\subsection{Symplectic duality for polarized hyperplane arrangements, or Gale duality}
		In the hypertoric setting, symplectic duality can be described in terms of an operation on polarised hyperplane arrangements known as Gale duality. Consider as above the sequence \ref{eq:basicsequence} of tori, together with a character $\eta$ of $\gaugegrp$. We also fix a cocharacter $\zeta$ of $T$. We define the Gale dual data to be
		\begin{enumerate}
			\item The set $\edges$.
			\item The dual sequence of tori 
			\begin{equation} \label{eq:dualbasicsequence} T^{\vee} \to D^{\vee} \to G^{\vee} \end{equation}
			with the induced isomorphism $D^{\vee} \cong (\C^\times)^{\edges}$.
			\item The character $-\zeta$ of $T^{\vee}$.
			\item The cocharacter $-\eta$ of $\gaugegrp^{\vee}$.
		\end{enumerate}
		Any construction starting from the first sequence and the parameters $\eta, \zeta$ may be performed starting from the second instead, using the parameters $-\zeta, -\eta$. We decorate the result with a shriek : $\arrangement^!, \quantumhyp^!$, etc. 
		
		Note that by definition we have $\edges = \edges^!$. In particular, the sets of sign vectors $\{ +, - \}^\edges$ for the Gale dual arrangements are canonically identified. Under this identification, the bounded and feasible chambers are exchanged. On the other hand, there is a natural bijection of the bases $\bases \cong \bases^!$ given by taking $b \subset \edges$ to its complement $b^c \subset \edges$.
		
		One of the main results of \cite{BLPWtorico} is that $\XX$ and $\XX^!$ are symplectically dual; we will spell this out in more detail below. 
		
		\begin{figure}[h]
			\centering{
				\resizebox{125mm}{!}{\includegraphics{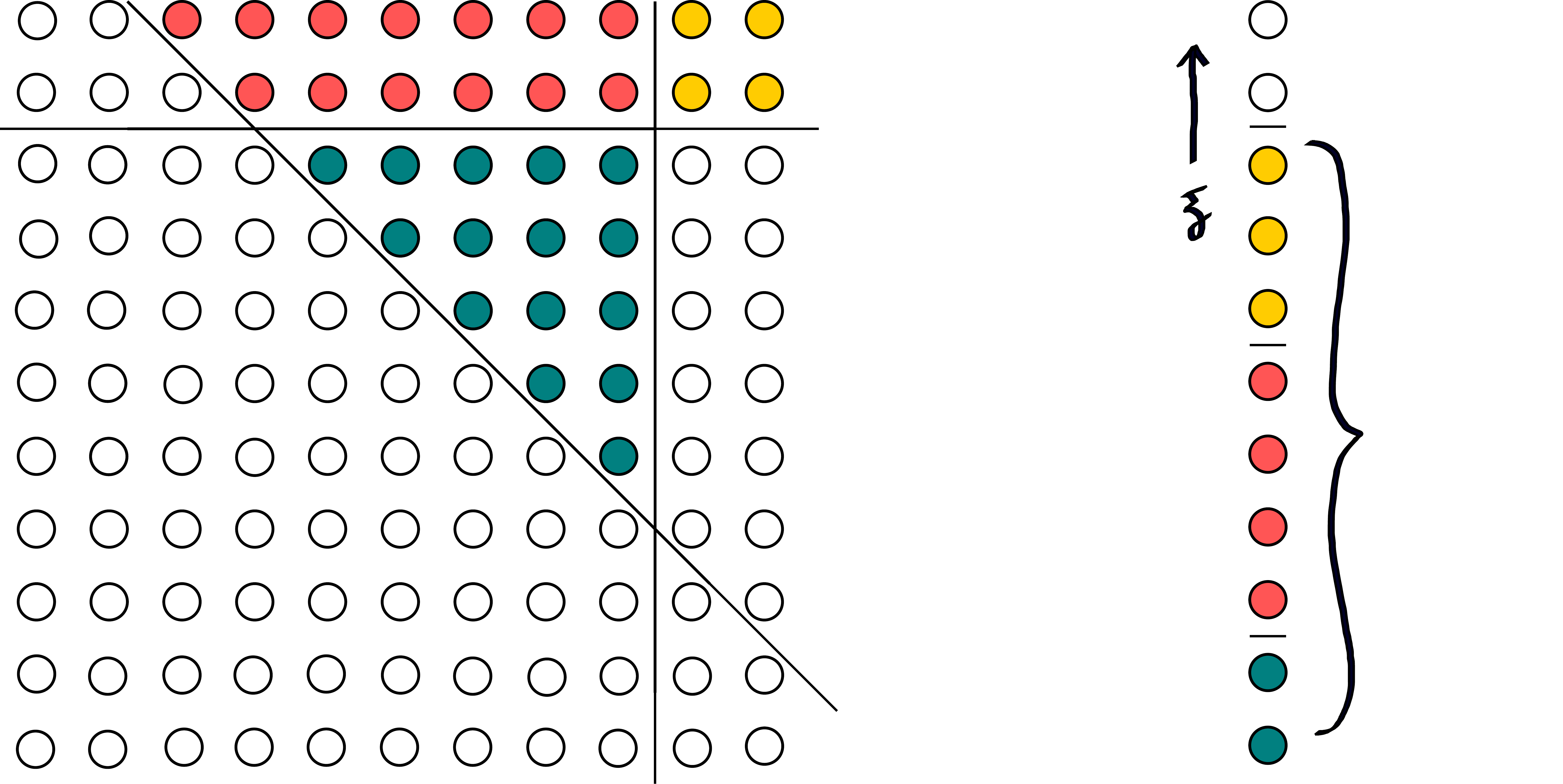}}
				\caption{The integer points of two Gale dual arrangements. The significance of these points will become clear in Section \ref{sec:GTmodules}. Chambers which correspond under the duality have been given matching colours. Colourless chambers appear only on one side of the duality. On the right, we have indicated the cocharacter $\zeta$ and the set of bounded chambers.}
				\label{fig:GDA}
			}
		\end{figure}
		
		\subsection{Hypertoric varieties from graphs, or cographical hypertorics}
		In this section we explain how to associate a hypertoric variety to any directed graph $\Gamma$. This class of examples includes many of  is often easier to grasp intuitively, while retaining most of the features of the general setting. We will take advantage of this intuitive presentation in our discussion of enumerative invariants.  
		
		Let $\edges$ be the edge set of $\Gamma$, and $V$ the set of vertices. Then we have a natural map of tori
		\[ (\C^{\times})^V = C^0(\Gamma, \C^\times) \xrightarrow{\partial} C^1(\Gamma, \C^{\times}) = (\C^\times)^{\edges} \] 
		given by the coboundary map. Let $\gaugegrp = (\C^{\times})^V / \C^\times$ be the quotient by the constant cochains; then we have a short exact sequence 
		\[\gaugegrp \to C^1(\Gamma, \C^{\times}) \to T := H^1(\Gamma, \C^\times). \]
		Pick a sufficiently generic character $\eta$ of $\gaugegrp$.
		\begin{definition}
			Let 
			\[ \XX_\eta(\Gamma) = T^{\vee} \C^{\edges} \sslash_\eta \gaugegrp. \]
			
		\end{definition} 
		$\XX_\eta(\Gamma)$ is by construction a hypertoric variety. We call hypertorics which arise in this way {\em cographical}. They are special cases of Nakajima quiver varieties, in which all of the vertices are given rank one.
		
		The Gale dual $\XX(\Gamma)^!$ of a cographical hypertoric is called graphical. When $\Gamma$ is planar, the Gale dual is the cographical hypertoric associated to the dual graph. 
		
		\begin{lemma} \label{lem:graphindexing}
			The vertices $b \in \bases$ are indexed by the spanning trees of $\Gamma$; more precisely, each $b \subset \edges$ is the set of edges {\em not} appearing in a spanning tree. The composition $\gauge_\Z \to \Z^{\edges} \to \Z^b$ of the coboundary map with the natural projection is thus an isomorphism, whose inverse is precisely $\phi_b : \Z^b \to \fg_\Z$.
		\end{lemma} 
		
		An important class of cographical hypertorics are obtained by `abelianizing' more general quiver varieties. Consider a quiver $Q$ with vertices $v_i$ of fixed rank $r_i$. 
		\begin{definition}
			Define the {\em abelianization} $Q^{\operatorname{ab}}$ to be the quiver obtained by splitting each $v_i$ into $r_i$ new vertices $v_i^{j}, j=1,...,r_i$, with a map between $v_i^j$ and $v_{i'}^{j'}$ for each map between $v_i$ and $v_{i'}$. 
		\end{definition}
		We view $Q^{\operatorname{ab}}$ as a directed graph. Given $\eta \in C_0(Q^{\operatorname{ab}}, \Z)$, we can form the cographical hypertoric $\XX_\eta(Q^{\operatorname{ab}})$.   The geometry and representation theory of $\XX_\eta(Q^{\operatorname{ab}})$ reflects that of the quiver variety attached to $Q$, while admitting a more combinatorial description \cite{MR2104010}. 
		
		Although the results of this paper are not specific to hypertoric varieties arising from abelianization, they are an important source of motivation for us.

		\section{Categories}
		In the following sections we will discuss various categories arising from the quantization of hypertoric varieties. We begin by establishing some general preliminaries on Koszul and highest-weight categories. The reader may wish to skip to Section \ref{sec:quantizedhypertoric} and return as needed to the previous sections.
		
		Roughly speaking, an abelian category $C$ is Koszul if each simple object admits a projective resolution $P^{\bullet} \to L$ that `looks like' the classical Koszul resolution. The complexes $P^{\bullet}$ are in turn the projective objects of a certain abelian subcategory $LPC(C)$ of the category of chain complexes in $C$. We say $LPC(C)$ is Koszul dual to $C$. The key feature for us will be an identification of Ext groups of simple objects on one side of the duality with Hom spaces of projective objects on the dual side. 
		
		We make this precise below, closely following the exposition in \cite{MR3594665}, to which we refer for further details. 
		
		\subsection{Mixed categories}
		Conside an abelian category $\tilde{C}$ with a choice of `weight' $wt(L) \in \Z$ for each simple object $L$, with finitely many simples in any given weight. $\tilde{C}$ is said to be mixed if whenever $wt(L) \leq wt(L')$ we have  $\Ext^{\bullet}(L,L') = 0$. As explained in \cite{BGS96}, one may think of this as being `graded semisimple'.  
		
		We suppose $\tilde{C}$ has a Tate twist, i.e. an automorphism denote $M \to M(1)$ on objects such that $wt(L(1)) = wt(L) - 1$. This allows us to define the category $\tilde{C} / \Z$ with the same objects as $\tilde{C}$, but graded morphism spaces $$\Hom_{\tilde{C}/\Z} (M, M') := \bigoplus_{d \in \Z} \Hom_{\tilde{C}}(M, M'(-d)).$$
		
		Let $P$ be the direct sum of projective covers of all simples of weight $0$ in $\tilde{C}$. Consider the graded ring \[ R := \Hom_{\tilde{C}/\Z}(P,P). \] The category $C$ of finite dimensional right-modules over $R$ is said to be a degrading of $\tilde{C}$; conversely, $\tilde{C}$ is said to be a graded lift of $C$. 
		
		\subsection{Koszul categories} 
		Mixed categories admit the following generalization of the classical Koszul resolution. Recall that the head of a module is its largest semisimple quotient. A projective resolution $P_{\bullet} \to M$ in $\tilde{C}$ is said to be {\em linear} if the heads of all indecomposable summands of $P_j$ have weight $j$.  
		\begin{definition} 
			An abelian category $C$ is said to be Koszul if it admits a graded lift for which any minimal projective resolution of a simple weight zero module is linear. A graded algebra is said to be Koszul if its category of graded modules is Koszul.
		\end{definition}
		We write $LPC(\tilde{C})$ for the category whose objects are linear projective complexes, and whose morphisms are chain maps; it is (somewhat surprisingly) abelian. The simple objects are indecomposable projectives supported in a single degree $j$. We can make $LPC(\tilde{C})$ into a mixed category by weighting these simples with weight $j$. 
		
		There is a functor $K_{\tilde{C}} : D^b(\tilde{C}) \to D^b LPC(\tilde{C})$ defined in \cite{MOS}, which is an equivalence precisely when the degrading of $\tilde{C}$ is Koszul. In this case, it takes indecomposable projectives to the corresponding simples. 
		
		Let $C$ and $C^!$ be Koszul, with graded lifts $\tilde{C}$ and $\tilde{C}^!$. 
		\begin{definition} \label{def:koszulduality}
			A Koszul duality between $C$ and $C^!$ is an equivalence of mixed categories 
			\[ \kappa: LPC(\tilde{C}) \to \tilde{C}^!. \]
		\end{definition}
		In particular, by precomposing with $K_{\tilde{C}}$, this defines a bijection between the indecomposable projectives $P_{\alpha}$ of $C$ and the simples $L^!_{\alpha}$ of $C^!$. 
		
		Let $L^! = \bigoplus_\alpha L^!_{\alpha}$ be the direct sum over all nonisomorphic simple objects in $C^!$, and let $P = \bigoplus_{\alpha} P_{\alpha}$ be the direct sum over all nonisomorphic indecomposable projective objects in $C$. Kozsul duality implies an equality
		\[ \Ext^{\bullet}(L^!,L^!) = \Hom(P,P) \]
		where the cohomological grading on the left-hand side corresponds to the grading induced by a graded lift of $P$ on the right, and the idempotents given by projection to any given $\alpha$-summand are identified.
		
		\subsection{Highest weight categories}
		Koszul duality plays well with the notion of highest weight categories, which appear frequently in representation theory. 
		\begin{definition}
			Let $C$ be a category with simples $L_{\alpha}$, projective covers $P_{\alpha}$ and injective hulls $I_{\alpha}$ indexed by $\alpha \in \indexset$, and let $\leq$ be a partial order on $\indexset$. $C$ is said to be highest weight if for each simple, there is an object $V_{\alpha}$ and epimorphisms
			\[ P_{\alpha} \to V_{\alpha} \to L_{\alpha} \]
			such that the kernel of the right-hand map is an extension of modules $L_{\beta}, \beta < \alpha$, whereas the kernel of the left-hand map is an extension of modules $V_{\gamma}, \gamma > \alpha$. We call $V_{\alpha}$ a standard object. 
		\end{definition}
		We will always further assume that $\End(V_\alpha) = \C$. If $C_{\leq \alpha}$ is the subcategory generated by $L_{\lambda}$ with $\lambda \leq \alpha$, then $V_{\alpha}$ may be characterised as the projective cover of $L_{\alpha}$ in $C_{\leq \alpha}$. We call the injective hull $\Lambda_{\alpha}$ of $L_\alpha$ in $C_{\leq \alpha}$ a {\em costandard} object.
		
		Yet a third class of objects will play an important role for us.
		\begin{definition}
			$T \in C$ is {\em tilting} if it admits a filtration by standard objects and a filtration by costandard objects. 
		\end{definition}
		One can show that indecomposable tilting objects are also indexed by $\indexset$, so that $T_{\alpha}$ has largest standard submodule $V_{\alpha}$ and largest costandard quotient $\Lambda_\alpha$.
		
		\subsection{Quantized hypertoric varieties} \label{sec:quantizedhypertoric}
		We now turn to certain Koszul categories arising from the quantization of hypertoric varieties. Consider the ring of differential operators 
		$D(\C^\coordset) = \C \langle z_e, \partial_e | e \in \edges \rangle$. It carries a natural homomorphism $$\Sym \fd \to D(\C^\coordset),$$ taking the $e$th coordinate element $\delta_e$ to the Euler operator $z_e \frac{\partial}{\partial z_e}$. We think of $D(\C^\coordset)$ as a quantization of $T^{\vee}\C^{\coordset}$, and the homomorphism as a quantization of the moment map $T^{\vee}\C^{\coordset} \to \fd^{\vee}$ for the action of $D$.
		
		Fix $\eta \in \frak{g}^{\vee}_\Z$ and let $\ker \eta$ be the kernel of the induced map $\Sym \fg \to \C$. Via the inclusion $\Sym \fg \to \Sym \fd$, we may view $\ker \eta$ as a subspace of $D(\C^\coordset)^{G}$.
		\begin{definition}
			The hypertoric enveloping algebra is given by 
			\[ \quantumhyp = D(\C^\coordset)^{G} / D(\C^\coordset)^G \langle \ker \qeta \rangle. \]
		\end{definition} 
		The definition of $\quantumhyp$ is a quantum analogue of Definition \ref{def:hypertoricreduction}. Indeed, the filtration of $\quantumhyp$ induced by the usual filtration on differential operators by order yields an associated graded algebra isomorphic to the coordinate ring $\mathcal{O}(\XX) = \mathcal{O}(\mu^{-1}(0))^G$. $U_\eta$ was studied in detail by Musson and Van den Bergh in \cite{musson1998invariants} in the more general context of rings of torus-invariant differential operators.

		The ring $\quantumhyp$ arises from a sheaf on $\XX_{\bar{\eta}}$, where $\bar{\eta}$ is another character of $\gaugegrp$. We recall this construction briefly below, for motivational purposes. Recall that the symplectic form on $\XX_{\bar{\eta}}$ gives its structure sheaf $\mathcal{O}_{\XX_{\bar{\eta}}}$ a Poisson bracket $\{ - , - \}$. The following makes sense for a general $\C^{\times}$-equivariant symplectic variety.
		\begin{definition} A quantization of  $\mathcal{O}_{\XX_{\bar{\eta}}}$ is a dilation-equivariant sheaf $\sheafQ$ of flat $\C[[\hbar]]$ algebras over $\XX_{\bar{\eta}}$, with $\hbar$ of dilation-weight n, together with a dilation-equivariant isomorphism $\sheafQ/\hbar \sheafQ \cong \mathcal{O}_\XX$, such that for any local sections $f,g$ of $\mathcal{O}_\XX$ and lifts $\tilde{f}, \tilde{g}$ to $\sheafQ$, the element $[\tilde{f}, \tilde{g}] \in \hbar\sheafQ$ has image $\{ f, g \}$ in $\hbar \sheafQ / \hbar^2\sheafQ \cong \mathcal{O}_\XX$. 
		\end{definition}
		Analogous quantizations, for $\XX$ a smooth symplectic variety (without $\C^\times$ action), were studied by De Wilde and Lecomte \cite{de1983existence} and Fedosov \cite{fedosov1994simple} in the smooth setting. Fedosov's methods were extended to the algebraic setting by Bezrukavnikov and Kaledin in \cite{bezrukavnikov2004fedosov}. For symplectic resolutions, a theorem due to Losev \cite{losev2012isomorphisms} identifies the space of $\C^\times$-equivariant quantizations with $\H^2(\XX_{\eta}, \C)$ via a certain `non-commutative period map'. The latter in turn equals $\frak{g}^{\vee}_\C$ for the hypertoric variety $\XX$. 
		
		We consider the sheaf $\sheafQ_\qeta$ associated to an integral (but otherwise generic) parameter $\qeta \in \frak{g}_\Z^{\vee}$. The algebra of global sections of $\sheafQ_\qeta$ is a $\C[[\hbar]]$-algebra with an action of $\C^{\times}$. We can consider the subalgebra of $\C^{\times}$-finite elements, and specialise $\hbar=1$ to obtain a finitely generated algebra over $\C$. This is precisely the algebra $\quantumhyp$. 
		
		It follows from the above that the global section ring of $\sheafQ_{\eta}$ does not depend on the choice of $\bar{\eta}$. On the other hand, the functor of global sections, taking sheaves of $\sheafQ_\qeta$-modules to $\quantumhyp$-modules, will be an equivalence of abelian categories only when $\eta$ and $\bar{\eta}$ are suitably compatible.
		
		In the next sections, we discuss some particularly nice subcategories of $U_\eta$ modules, which we will eventually relate to enumerative invariants. 
		
		\subsection{Gelfand-Tsetlin modules} \label{sec:GTmodules}

		\begin{definition}
			The Gelfand-Tsetlin category $\cG_{\eta}$ of $\quantumhyp$ is the category of finitely generated modules $M$ over $\quantumhyp$ such that we have an $\Sym \ft$ module decomposition 
			\begin{equation} \label{eq:weightdecomp} M = \bigoplus_{\frak{m}  \in \ft^{\vee}} M[\frak{m}], \end{equation} where the action of $\Sym \ft$ on $M[\frak{m}]$ factors through $\Sym \ft / \frak{m}^k$ for some $k \geq 0$.  Here $\frak{m}$ is a maximal ideal of $\Sym \ft$. \end{definition}
		
		In order to describe this category more explicitly, it is useful to consider the {\em support} of a $\quantumhyp$-module $M$ in $\frak{t}^{\vee}_\Z$, meaning the set of weights which appear in the decomposition \ref{eq:weightdecomp}. The support of any fixed module equals the lattice points of a certain polytope, obtained by `quantizing' the constructions of Section \ref{sec:qarrangements}. Namely, to each sign vector $\alpha \in \{ +, - \}^{\edges}$ we associate the set of lattice points $\qDelta_\alpha = \{ d_e \in \Z^{\geq 0} | \alpha(e) = + \} \cap \{ d_e \in \Z^{\leq -1} |  \alpha(e) = - \} \cap \frak{t}_{\eta}^{\vee}$. We define the sets $\qfeasible_\eta$ and $\qbounded^\zeta$ of feasible and bounded chambers as in Definitions \ref{def:feasible} and \ref{def:bounded}, replacing $\Delta_\alpha$ by $\qDelta_\alpha$.
		
		\begin{lemma}
			The support of any module $M \in \cG_{\eta}$ is a union of feasible chambers $\qDelta_\alpha, \alpha \in \qfeasible_\eta$.
		\end{lemma}

		\begin{definition}
			Let $\mu \in \ft_{\eta}^{\vee}$. Define $P^{(k)}_\mu := \quantumhyp / \langle \frak{m}^k_{\mu} \rangle$ and let $P_\mu$ be the projective limit of $P^{(k)}_\mu$ as $k$ tends to infinity.
		\end{definition}
		\begin{proposition} \label{prop:projectivesareweightfunctors}
			We have an equality of vector spaces
			\begin{equation} \label{eq:projectivehomisweight} \Hom(P_\mu, M) = M[\mu]. \end{equation}
			Choose an element $\mu \in \qDelta_{\alpha}$ for each feasible $\alpha$. The modules $P_{\mu}$ are a complete and irredundant set of indecomposable projective pro-objects in $\cG_\eta$. 
		\end{proposition}
		
		\begin{proof}
			This is proven in \cite{musson1998invariants}; our notation is closer to \cite[Section 3.4]{GDKD}. Equation \ref{eq:projectivehomisweight} is a direct consequence of the definition of $P_\mu$. Since taking weight spaces is an exact functor, it follow that $P_\mu$ is projective. By hypothesis, Gelfand-Tsetlin modules are direct sums of their weight-spaces, from which it follows that the $P_\mu$ form a complete set of projectives as $\mu$ ranges over the weights of $T$. Determining when $P_\mu$ and $P_{\mu'}$ are isomorphic is the most delicate part of the proof; for $\mu, \mu'$ in the same chamber, the isomorphism is constructed from the action of the weight space $\quantumhyp [\mu' - \mu]$. 
		\end{proof}
		
		\begin{definition} 
			Let $L_\alpha$ be the simple quotient of $P_{\alpha}$. 
		\end{definition}
		The $L_{\alpha}$ for feasible sign vectors $\alpha$ form a complete and irredundant set of simple modules in $\cG_\eta$. These simple objects are `quantizations' of the lagrangians $\frak{L}_{\alpha}$. A concrete manifestation of this is Proposition \ref{cor:extsarecohomology}, which the reader may want to read immediately before proceeding. 
		
		\subsection{Quiver algebras}
		The description of projective objects as weight functors in Proposition \ref{prop:projectivesareweightfunctors} implies the following  handy description of the category $\cG_\eta$. 
		
		Let $Q_E$ be the quiver algebra with indempotents $e_\alpha$ for each $\alpha \in \{ +, - \}^E$, an edge between vertices that differ by a single sign (in particular, edges in both directions), and relations imposing the equality of two-edge paths with the same start and endpoints. Let $\hat{Q}_E$ be the completion with respect to the grading by path length. 
		
		The significance of $\hat{Q}_E$ stems from the following.
		\begin{proposition} \label{prop:Qcontrols} \cite[Proposition 3.5.6 and Theorem 6.3]{musson1998invariants} The category of finitely generated $D(\C^\coordset)$-modules which decompose as direct sums of generalized eigenspaces for the action of $\Sym \fd$ is equivalent to the category of finite dimensional modules over $\hat{Q}_E$.
		\end{proposition}
		
		Let $\theta : \frak{d} \to Q_E$ be the map taking $\delta_e$ to $\theta_{e}$, where the latter is the sum over all $\alpha \in \{ +, - \}^E$ of the two-edge composition which flips the $e$th coordinate twice. \begin{definition} \label{def:Ralgebra}
			Let \[ R := Q_E / \theta(\frak{g}) \] 
			and let $\hat{R}$ be the completion of $R$ with respect to the length filtration.
			For $\eta \in \frak{g}^{\vee}$, let $e_\eta := \sum_{\alpha \in \qfeasible_\eta} e_\alpha$. Define
			\[  R_{\eta} := e_\eta R e_\eta \text{  and  } \hat{R}_\eta := e_\eta \hat{R} e_\eta. \]
		\end{definition}
		Let $P := \bigoplus_{\alpha \in \feasible_\eta} P_{\alpha}$ be the sum of all indecomposable projective modules.
		\begin{proposition} \label{thm:quiverforGT}
			We have
			\begin{equation} \label{eq:Risprojectiveend} \hat{R}_{\eta} = \End(P). \end{equation}
			The functor taking $M \in \cG_\eta$ to $\Hom(P, M)$ is an equivalence between $\cG_\eta$ and the category of finite dimensional modules over $\hat{R}_{\eta}$.
		\end{proposition}
		\begin{proof}
			This can be derived directly from Proposition \ref{prop:Qcontrols}. In terms of the description of $\cG_\eta$ established in the previous section, however, we can understand this equivalence as follows. By Proposition \ref{prop:projectivesareweightfunctors}, $P$ is a projective generator of $\cG_\eta$. Equation \ref{eq:projectivehomisweight} may be used to show $\End(P) = \hat{R}_{\eta}$, as in \cite[Theorem 3.1.7]{musson1998invariants}. 
		\end{proof}
		
		\subsection{The dual of $\cG_{\eta}$} 
		In this section, we study the Koszul dual of $\cG_\eta$ from an algebraic perspective. In the following section, we will relate it to the symplectic dual $\XX^!$. 
		
		Consider the algebra 
		\begin{equation} \label{eq:algebra2} R^\zeta := R/Re_\zeta R  \end{equation}
		and its completion $\hat{R}^{\zeta}$. Let $\cG^{\zeta}$ be the category of finite dimensional $(\hat{R}^!)^{\zeta}$-modules. Just as the simples of $\cG_\eta$ were indexed by the $\eta$-feasible $\alpha \in \{+ , - \}^{\edges}$, the simple modules of $\cG^{\zeta}$ are indexed by the $\zeta$-bounded $\alpha$.

		Let $(R^!)^{-\eta}$ denote the algebra \ref{eq:algebra2} attached to the Gale dual arrangement \ref{eq:dualbasicsequence}. 
		\begin{theorem} \label{thm:Koszulbigalgeras}
			$R_{\eta}$ and $(R^!)^{-\eta}$ are Koszul dual algebras. 
		\end{theorem}
		\begin{proof}
			This is \cite[Lemma 8.25]{BLPWtorico}. The proof amounts to a rather intricate calculation showing that a certain `Koszul complex' with underlying vector space $R^{-\eta} \otimes R^!_{\eta}$ is exact. 
		\end{proof}
		
		The key features of this Koszul duality for us are Corollaries \ref{cor:extsonthedual} and \ref{cor:extsarecohomology}. 
		
		\begin{corollary} \label{cor:extsonthedual}
			Let $L_{\alpha_1}, L_{\alpha_2}$ be simple modules in $\cG_{\eta}$. Then 
			\begin{equation} \label{eq:koszultwistsimproj} \Ext^{\bullet}(L_{{\alpha_1}}, L_{{\alpha_2}}) = e_{\alpha_1} (R^!)^{-\eta} e_{\alpha_2}. \end{equation}
		\end{corollary}
		
		We say that $\bar{\eta}$ is {\em linked} to $\eta$ if $\feasible_{\bar{\eta}} = \qfeasible_{\eta}$. Let $\bar{\eta}$ and $\eta$ be linked, and let $L_{\alpha_1}, L_{\alpha_2}$ be two simple modules. Recall that $\alpha_1, \alpha_2$ also parametrize lagrangian subvarieties $\frak{L}_{\alpha_1}, \frak{L}_{\alpha_2} \subset \XX_{\bar{\eta}}$, as in Section \ref{sec:lagsfromchambers}. The modules $L_{\alpha_i}$ are, in a suitable sense, quantizations of these lagrangians. In particular, we have the following.
		\begin{corollary} \label{cor:extsarecohomology}
			We have an isomorphism of graded vector spaces
			\begin{equation} \label{eq:extisintersect} \H^{\bullet - \codim}(\frak{L}_{\alpha_1} \cap \frak{L}_{\alpha_2}, \C) \cong \Ext^{\bullet}( L_{\alpha_1}, L_{\alpha_2}). \end{equation}
		\end{corollary}
		Here $\codim$ is the complex codimension of $\frak{L}_{\alpha_1} \cap \frak{L}_{\alpha_2}$ in $\frak{L}_{\alpha_2}$. The Yoneda product is given by a simple convolution rule, which we do not discuss here.

		\begin{proof}
			This follows from Corollary \ref{cor:extsarecohomology} and \cite[Theorem 6.1]{GDKD}, which shows that the weight spaces $e_{\alpha_1}(R^!)^{-\eta}e_{\alpha_2}$ are given by the left-hand side of Equation \ref{eq:extisintersect}.
			
			Heuristically, this identification may be understood as follows. The modules $L_{\alpha_i}$ may be constructed by taking global sections of a sheaf of modules $\cL_{\alpha_i}$ over $\cQ_{\eta}$ supported along $\frak{L}_{\alpha_i}$. Locally along $\frak{L}_{\alpha_1}$, the sheaf $\cQ_\eta$ ressembles a sheaf of twisted differential operators, and the modules $\cL_{\alpha_i}$ ressemble regular holonomic modules over this sheaf. Such modules may be identified with constructible sheaves via the Riemann-Hilbert correspondence. In our particular case, this identifies the Ext groups (up to a shift) with $\Ext(\C_{\frak{L}_{\alpha_1}}, \C_{\frak{L}_{\alpha_1} \cap \frak{L}_{\alpha_1} }) = \H^{\bullet}(\frak{L}_{\alpha_1} \cap \frak{L}_{\alpha_2}, \C)$, where Exts are taken in the constructible category of $\frak{L}_{\alpha_1}$. 
		\end{proof}

		We can write the shift by codimension more explicitly via the following.
		\begin{definition} \label{def:hamming}
			Given chambers $\alpha, \beta \in \bounded^{\zeta}$, let $d_{\alpha, \beta}$ be the length of the shortest path in $\bounded^{\zeta}$ from $\alpha$ to $\beta$. In other words, it is the minimal number of signs one can flip in $\{+, - \}^{\edges}$ to get from $\alpha$ to $\beta$ without leaving $\bounded^{\zeta}$. 
			
			Geometrically, $d_{\alpha, \beta}$ is the codimension of $\frak{L}^!_{\alpha} \cap \frak{L}^!_{\beta} \subset \frak{L}^!_{\alpha}$, when this intersection is non-empty.	
		\end{definition}
		
		\subsection{Category $\cO$} \label{sec:qarrangements}
		A certain subcategory of $\cG_\eta$ called category $\cO$ plays a key role in the original definition of symplectic duality, and will play a starring role in this paper.
		
		For each $\zeta \in \flavourlie$, we have a decomposition $\quantumhyp = \quantumhyp^+ \oplus \quantumhyp^0 \oplus \quantumhyp^-$ into elements scaled positively, fixed or scaled negatively by $\zeta$. 
		\begin{definition}
			$\cO_{\eta}^{\zeta}$ or `algebraic category $\cO$' is defined to be the abelian category of finitely generated $\quantumhyp$-modules on which $\quantumhyp^+$ acts locally finitely.
		\end{definition}
		
		For each such $\zeta$, we have a left adjoint to the natural inclusion $\cO_{\eta}^{\zeta} \to \cG_{\eta}$ given by the $\zeta$-truncation functor $\pi_{\zeta} : \cG_{\eta} \to \cO_{\eta}^{\zeta}$. This takes a module $M$ to its quotient by the subspace generated by $M[\mu]$ for $\mu$ lying in a $\zeta$-unbounded chamber. One can show that $\pi_\zeta(P_{\alpha})$ is an indecomposable projective of $\cO_\eta^{\zeta}$, nonzero exactly when $\alpha \in \qbounded^\zeta \cap \qfeasible_\eta$. With some more work, one obtains the following, proven in \cite{BLPWtorico}:
		
		\begin{proposition} 
			$\cO_{\eta}^{\zeta}$ is a highest weight category, with index set $\indexset$ given by $\qbounded^\zeta \cap \qfeasible_\eta$.  
		\end{proposition}
		
		For each $\alpha \in \qbounded^\zeta \cap \qfeasible_\eta$, we denote by
		\[ L_\alpha, P_{\alpha}, V_{\alpha}, T_{\alpha} \]
		the simple, projective, standard and tilting modules in $\cO_{\eta}^{\zeta}$ indexed by $\alpha$. 
		
		We fix $\eta \in \frak{g}^{\vee}_\Z$ and $\zeta \in \frak{t}_\Z$, neither contained in a root hyperplane. One of the main features of symplectic duality for hypertorics is the following result, proven in \cite[Corollary 4.20]{BLPWtorico} using the results of \cite{GDKD}. 
		\begin{theorem} \label{thm:classicKoszul} 
			The category $\cO_\eta^{\zeta}$ is standard Koszul. Its Koszul dual is the category $\cO^{-\eta}_{-\zeta}$ associated to the Gale dual data \ref{eq:dualbasicsequence}. 
		\end{theorem}
		Thus the quantizations of symplectically dual hypertorics $\XX_\eta$ and $\XX_{-\zeta}^!$ are, in a suitable sense Koszul dual to each other. %We write $\tilde{\cO}_\eta^{\zeta}$ for the graded lift of $\cO_\eta^{\zeta}$ used in Theorem \ref{thm:classicKoszul}.

		As with $\cG_\eta$, there is a quiver description of category $\cO$. Let $e_\zeta = \sum_{\alpha \notin \qbounded^\zeta} e_\alpha$. We have the algebra
		\begin{equation} \label{eq:algebra1} R_\eta^\zeta := R_\eta/R_\eta e_\zeta R_\eta  \end{equation}
		and its completed analogue $\hat{R}_{\eta}^{\zeta}$.
		\begin{proposition} \label{thm:quiverformatrixO}
			$\cO_\eta^{\zeta}$ is equivalent to the category of finite dimensional modules over $\hat{R}_\eta^{\zeta}$.
		\end{proposition}
		\begin{proof}
			This is (the integral case of) \cite[Theorem 4.7]{BLPWtorico}.
		\end{proof}

		\subsection{Twisting functors and Ringel duality}
		This section reinterprets the right-hand side of Equation \ref{eq:koszultwistsimproj} in terms of the quantization of $\XX^!$. The answer will be given in Proposition \ref{prop:crossedweightsandtiltings}. We begin with some preliminaries. 
		
		Given two generic $\eta_1, \eta_2 \in \fg^{\vee}$, \cite[Proposition 6.1]{GDKD} defines a twisting functor $_{\eta_1} \Phi_{\eta_2} : \cO^{\zeta}_{\eta_1} \to \cO^{\zeta}_{\eta_2}$. In terms of $R^{\zeta}_\eta$-modules, this is given by the derived functor of
		\[ M \to e_{\eta_1} R^{\zeta} e_{\eta_2} \otimes_{R^{\zeta}_{\eta_2}} M. \]
		When $\eta_2 = - \eta_1$, this functor is closely related to Ringel duality. In particular, it takes indecomposable projectives to indecomposable tilting modules \cite[Theorem 6.10]{GDKD}. More explicitly, Lemma \ref{lem:chamberstovertices} gives bijections
		\begin{equation} \label{eq:flipbijections} \qfeasible_{\eta_1} \cap \qbounded^{\zeta} \xleftarrow{\mu_1} \bases \xrightarrow{\mu_2}  \qfeasible_{\eta_2} \cap \qbounded^{\zeta} \end{equation}
		associating to $b \in \bases$ the bounded feasible chamber with $\zeta$-maximal point $H_b$. Let $\nu : \qfeasible_{\eta_1} \cap \qbounded^{\zeta} \to \qfeasible_{\eta_2} \cap \qbounded^{\zeta}$ be the composition. Setting $\eta_1 = - \eta_2$, the image of a projective module under the twisting functor is then given by
		\begin{equation} \label{eq:ringelonprojective} _{\eta} \Phi_{-\eta}(\smallproj_\alpha) = \tiltingnoshriek_{\nu(\alpha)}. \end{equation}
		We have the following formula for $\nu$.
		\begin{lemma}
			Let $\mu_1(b) = \alpha$. Then $\nu(\alpha)(e) = \alpha(e)$ if $e \notin b$, and $\nu(\alpha)(e) = - \alpha(e)$ otherwise.
		\end{lemma}
		
		Now suppose $\alpha_1 \in \qfeasible_\eta$ and $\alpha_2 \in \qfeasible_{-\eta}$. Then we have
		$e_{\alpha_2} R^{\zeta} e_{\alpha_1} = e_{\alpha_2} e_{-\eta} R^{\zeta} e_{\eta} e_{\alpha_1}. $
		
		As in the proof of \cite[Lemma 6.4]{GDKD}, we have $$e_{-\eta} R^{\zeta} e_{\eta} e_{\alpha_1} = {}_\eta \Phi_{-\eta}(P_{\alpha_1}).$$
		Combining with Equation \ref{eq:ringelonprojective} and keeping track of the natural gradings on either side, we obtain the following. % an equality of ungraded spaces
		%$$e_{\alpha_2} R^{\zeta} e_{\alpha_1} = e_{\alpha_2} \tiltingnoshriek_{v({\alpha_1})}.$$
		%After taking into account the natural $\Z$-gradings, we obtain
		\begin{proposition} \label{prop:crossedweightsandtiltings}
			\begin{equation} \label{eq:crossedweightsandtiltings} e_{\alpha_2} R^{\zeta} e_{\alpha_1}\langle -d_{\alpha_1, \nu(\alpha_1)} \rangle = e_{\alpha_2} \tiltingnoshriek_{v({\alpha_1})}. \end{equation}
			where the angle brackets denote a shift of the $\Z$-grading.
		\end{proposition}
		
		\begin{figure}[h]
			\centering{
				\resizebox{100mm}{!}{\includegraphics{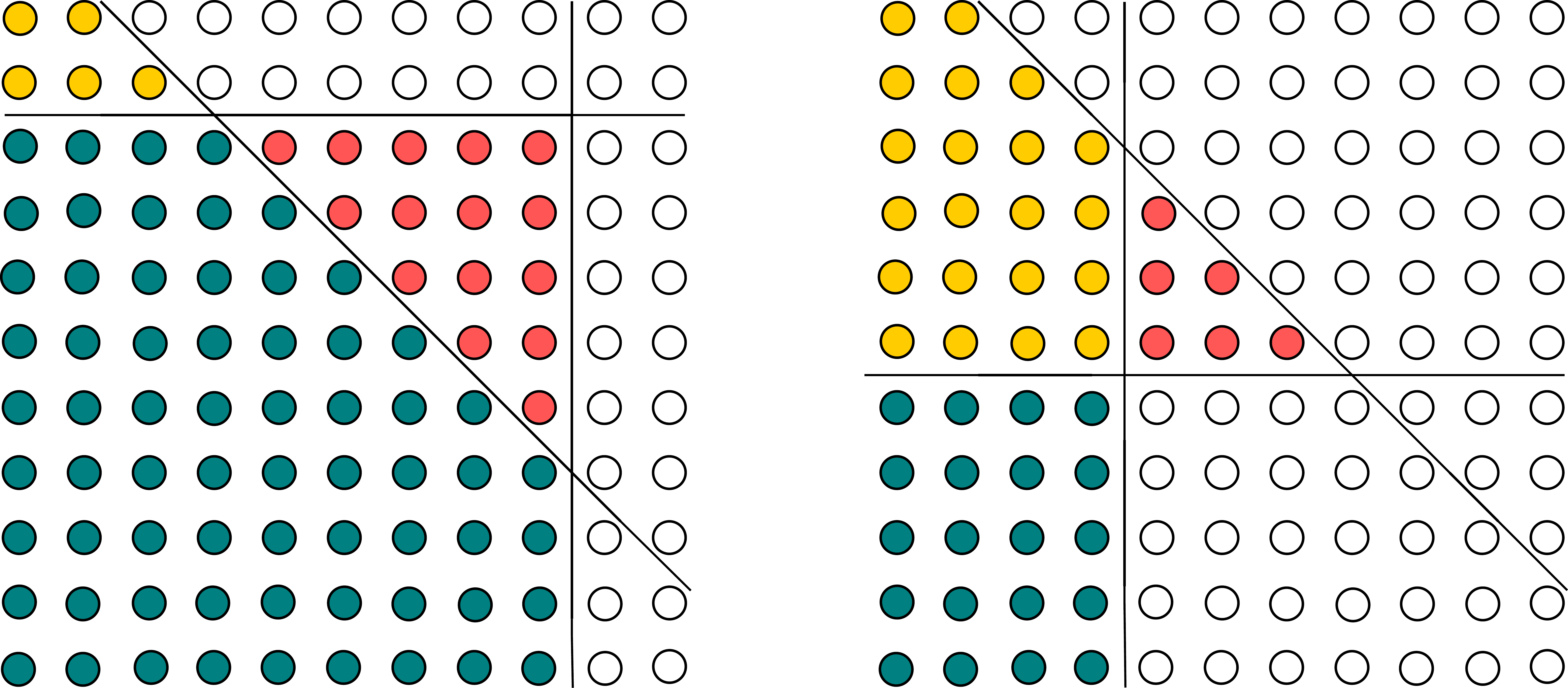}}
				\caption{We have drawn the arrangements associated to $\eta$ and $-\eta$, for a generic choice of $\eta$, in our favorite example. The central chamber, colored red, is feasible for only one of $\pm \eta$, whereas all other chambers are feasible for both choices. Chambers identified by $\nu$ have been given matching colors. We have left the unbounded chambers (with respect to a generic choice of cocharacter) colorless, for comparison with the Gale dual figures in Figure \ref{fig:twoetsdual}.}
				\label{fig:twoets}
			}
		\end{figure}
		
		\begin{figure}[h]
			\centering{
				\resizebox{18mm}{!}{\includegraphics{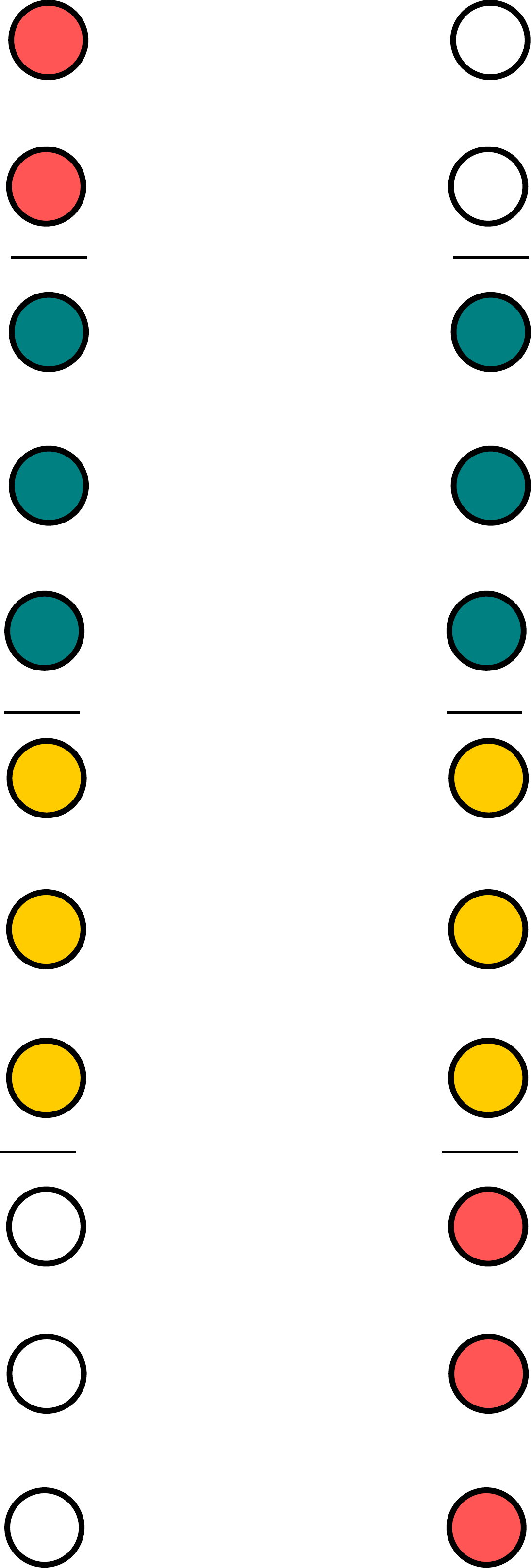}}
				\caption{We have drawn the Gale dual arrangements to Figure \ref{fig:twoets}, corresponding to opposite cocharacters. The red chambers are bounded for one choice of cocharacter and not the other, whereas all other chambers are bounded for both choices. Chambers identified by $\nu$ have been given matching colors.} 
				\label{fig:twoetsdual}
			}
		\end{figure}
		We will eventually be concerned with the class of this tilting module  in the Grothendieck group of category $\cO$, which can be understood in terms of classes of Verma modules as follows. Given $b \in \bases$, let $\bounded^{\zeta}_b = \{ \alpha : \alpha(e) = \mu_1(b)(e) \text{ for all } e \in b \}$. It indexes the chambers contained in the cone with vertex $H_b$, emanating in the $\zeta$-negative direction. 
		
		There is a partial order on $\bases$ generated by the relations $b < b'$ if $b$ and $b'$ differ by a single coordinate, and $\zeta(H_b) < \zeta(H_{b'})$. In other words, $b < b'$ if one can travel from $H_b$ to $H_{b'}$ along line segments of the arrangement in the $\zeta$-positive direction.
		
		\begin{proposition} \label{prop:tiltingfiltration}  \cite[Lemma 6.4]{GDKD}
			There is a filtration of $\tiltingnoshriek_{\nu(\beta)}$ indexed by the poset $\bases$. The associated graded space $F_b \tiltingnoshriek_{\nu(\beta)} / F_{> b} \tiltingnoshriek_{\nu(\beta)}$ of this filtration is isomorphic to the Verma module $V_{\mu_2(b)}$ if $\beta \in \bounded^{\zeta}_b$, and vanishes otherwise. 
			
		\end{proposition}
		
		This proposition is useful because the weight spaces of Verma modules are easy to write down. 
		\begin{lemma} \label{lem:vermaweightspaces}
			Let $b \in \bases$ and let $\alpha \in \feasible_\eta$. Then $V_{\mu(b)}[\alpha] \cong \C$ if $\alpha(e) = \mu(b)(e)$ for $e \in b$, and vanishes otherwise. 
		\end{lemma}
		\begin{lemma} \label{lem:vermaweightspaces}
			With respect to the $\Z$-grading on $T_{\nu(\alpha_1)}$, the subquotient weight space  $V_{\mu_2(b)}[\alpha]$ lies in degree $d_{\alpha_1, \mu_2(b)} + d_{\mu_2(b), \alpha} - d_{\alpha_1, \nu(\alpha_1)}$.  
		\end{lemma}
		\begin{proof}
			%This follows from Proposition \ref{prop:crossedweightsandtiltings} and 
			This follows from the explicit construction of the Verma modules in \cite[Lemma 6.4]{GDKD}. An element of the Verma module is represented by a path in the quiver concatenated from a taut path from $\alpha_1$ to $\mu_2(b)$ and a taut path from $\mu_2(b)$ to $\alpha$. The lengths of these paths give the first two terms. The last term is the degree shift in Proposition \ref{prop:crossedweightsandtiltings}.
		\end{proof}
		
		\section{Twisted quasimaps and enumerative invariants}
		
		We now turn from quantizations of hypertoric varieties to their enumerative geometry. The connection between these two seemingly unrelated topics will be established in Section \ref{sec:presmoduli}, where we begin to show that the enumerative invariants attached to $\XX_\eta$ can be expressed in terms of the quantization of a much larger hypertoric variety.
		
		We begin by recalling some general definitions and results from \cite{ciocan2014stable}. Let $W$ be an affine variety with local complete intersection singularities, together with an action of a reductive group $G$, linearized by a $G$-equivariant line bundle $\cL$. Let $W^{s} \subset W$ be the semistable locus for this linearization, which we assume equals the stable locus, so that the GIT quotient is $W \sslash_{\cL} G = W^{s} / G$. Let $[W/G]$ be the stack quotient. 
		
		Let $C$ be a fixed curve with at worst nodal singularities. In fact, we will work exclusively with $C \cong \P^1$ in the following. Let $f : C \to [W/G]$ be a map; we say it is a quasimap to $W \sslash_{\cL} G$ if the preimage of the unstable locus $W \setminus W^{ss}$ is a finite set.

		Twisted quasimaps, on the other hand, count sections of a certain $[W/G]$-bundle over $C$. Suppose we have an action of a torus $\mathbf{T}$ on $W$ commuting with $G$, and let $\cT$ be a $\mathbf{T}$-bundle on $C$. Let $\cW = W \times_{\mathbf{T}} \cT$. Then we may define twisted quasimaps as sections of the bundle $[\cW / G] \to C$ satisfying an analogous condition to the above.  
		\subsection{Twisted quasimaps to hypertoric varieties}
		We will be interested in twisted quasimaps to hypertoric varieties. For clarity, we focus on the case where $\XX_\eta = \XX_\eta(\Gamma)$ is cographical; this is a special case of Kim's construction in \cite{kim2016stable}, which in turn may be understood in terms of the construction outlined in the previous section. We leave the general case to the interested reader. 
		
		Let $\Gamma$ be a graph with vertices $\vertices$ and edges $\edges$. We pick a distinguished `framing vertex' $\framevertex \in \vertices$. 
		
		Fix a curve $C$, and for each edge $e$ of $\Gamma$ a pair of line bundles $\mass^e_x, \mass^e_y$ on $C$ such that $\mass^e_x \otimes \mass^e_y \cong \omega_C^{-1}$. Fix $\lambda_v \in H^0(C, \omega_C)$ for each vertex $v$ of $\Gamma$.
		
		\begin{definition} \label{def:twistedquasimap}
			A twisted quasimap $\qs$ associates to each vertex $v$ of $\Gamma$ a line bundle $\qs_v$ on $C$, and to each edge a pair of maps $$x_e :  \mass^e_x \otimes \qs_{t(e)} \to \qs_{h(e)} $$ and $$y_e : \mass^e_y \otimes \qs_{h(e)} \to \qs_{t(e)}$$ of the underlying coherent sheaves, satisfying moment map relations indexed by the vertices $v \neq \framevertex$ of $\Gamma$: 
			\begin{equation} \label{eq:momentcondition} \sum_{h(e) = v} x_e y_e - \sum_{t(e) = v} y_e x_e = \lambda_v. \end{equation}
		\end{definition}
		Recall that $\XX_\eta(\Gamma) = \mu_\gaugegrp^{-1}(0) \sslash_\eta \gaugegrp$.
		Choose nonvanishing sections of $\mass_x^e$ and $\mass_y^e$ over some open $U \subset C$. Then such a quasimap defines a map $U \to [\mu_\gaugegrp^{-1}(0) / \gaugegrp ]$. 
		\begin{definition}
			We say a twisted quasimap is stable if for a trivialization as above over some open dense $U \subset C$, the resulting map has image in the stable locus of $\mu_\gaugegrp^{-1}(0)$.
		\end{definition}
		From now on, we fix $C = \P^1$ and thus $\lambda_v = 0$. The set of degrees $(\deg(\mass^e_x), \deg(\mass^e_y))$ form a torsor $\twistspace$ over $\fd_\Z$. We have an isomorphism $\twistspace \cong \fd_\Z$ which associates to $\twist \in \fd_\Z$ the pair of bundles $\mass_x^e = \cO((\twist^e + 1) \cdot \infty), \mass_y^e = \cO((-\twist^e + 1) \cdot \infty)$. 
		
		We write $\QM_\twist(\P^1, X_\eta, \degree)$ for the moduli of stable twisted quasimaps of degree $\degree \in \H_2(\XX_\eta, \Z)$. This is a mild abuse of notation, as the moduli depends on the presentation of $\XX_\eta$ as a GIT quotient. It is a finite type Deligne-Mumford stack, proper over $\Spec H^0(\XX_\eta, \mathcal{O}_{\XX_\eta})$.

		\begin{lemma}
			Let $\qs \in \QM_\twist(C, \XX_\eta, \degree)$ and let $C \cong \P^1$. Then at most one arrow along each edge $e \in \edges$ carries a non-zero morphism.  
		\end{lemma}
		\begin{proof}
			The compositions $x_e y_e$ and $y_e x_e$ both lie in $H^0(C, (\mass^e_x \otimes \mass^e_y)^{-1}) = H^0(C, \omega_C)$, which vanishes.
		\end{proof}
		
		We now give a more concrete description of this moduli space. Fix a degree $\degree \in \H_2(\XX, \Z) \cong \fg_\Z$, and consider the vector space \[ \H(\gamma) := \bigoplus_{e \in \edges} \Hom(\mass^e_{x} \otimes \mathcal{O}(\degree_{h(e)}), \mathcal{O}(\degree_{t(e)})) \oplus\Hom(\mass^e_{y} \otimes \mathcal{O}(\degree_{t(e)}), \mathcal{O}(\degree_{h(e)})). \]
		It carries an action of $(\C^\times)^{V} = \gaugegrp$. We view $\mathcal{O}(n) = \mathcal{O}(n \cdot \infty)$ as the sheaf of functions on $\C$ of degree at most $n$ at $\infty$. Thus there is a $\gaugegrp$-equivariant `evaluation' map $$ \ev: \C^\times \times \H(\gamma) \to T^{\vee} \C^{\edges},$$ taking $(c, h_e \oplus \bar{h}_e)$ to $(h_e(c), \bar{h}_e(c))$. 
		
		The character $\eta \in \fg_\Z^{\vee}$ determines a $\gaugegrp$-equivariant structure on the line-bundle $\mathcal{O}_{\H(\gamma)}$, and thus a linearization of the $\gaugegrp$-action.
		\begin{lemma} \label{lem:stabtostab}
			A point $h \in \H(\gamma)$ is semistable if and only if $\ev(\C^\times \times h)$ meets the semistable locus of $T^{\vee}\C^{\edges}$. 
		\end{lemma}
		\begin{proof}
			For any $c$, $\ev(c \times -)$ is a $\gaugegrp$-equivariant map of vector spaces preserving linearizations, and thus maps the unstable locus into the unstable locus. This takes care of one direction. For the converse, suppose that $\ev(c \times h)$ is unstable for all $c$. The unstable locus of $T^{\vee}\C^{\edges}$ is a union of coordinate subspaces; by irreducibility of $\C^\times$, $\ev(\C^\times \times h)$ must lie entirely in one of these. One can check directly that in such cases, $h$ is unstable. \end{proof}
		
		Let $\cE \to \QM_\twist(\P^1, \XX_\eta, \degree )$ be the principle $\gaugegrp$-bundle parametrizing choices of isomorphisms $\sheaf_v \cong \mathcal{O}(\degree_v)$. There is a tautological $G$-equivariant embedding $\cE \to H(\degree)$, which descends to an embedding $$\rho : \QM_\twist(\P^1, \XX_\eta, \degree ) \to [\H(\degree) / \gaugegrp].$$ Moreover, the map $\ev$ descends to the natural map $\overline{\ev} : \C^\times \times \QM_\twist(\P^1, \XX_\eta, \degree) \to [T^{\vee} \C^{\edges} / \gaugegrp]$.

		\begin{lemma} \label{lem:quasimapsareGITquotient}
			The map $\rho$ is an isomorphism of $\QM_\twist(\P^1, \XX_\eta, \degree )$ onto the GIT quotient
			\[  \H(\gamma) \sslash_\eta \gaugegrp. \]
		\end{lemma}
		\begin{proof}
			We must show that $\rho$ identifies quasimap stability with GIT stability. This follows from Lemma \ref{lem:stabtostab}, since a quasimap $\sheaf$ is stable if and only if there exists $c \in \C^\times$ for which $\overline{\ev}(c \times \sheaf)$ lies in the stable locus of $T^{\vee} \C^{\edges}$.\end{proof}

		\subsection{Obstruction theories}
		
		In our setting, the moduli of twisted quasimaps to a quiver varieties carries a perfect obstruction theory. Consider the sheaf on $\P^1 \times \QM_\twist(\mathbb{P}^1, \XX_\eta, \degree)$ given by 
		\[ \tgtquasimaps := \bigoplus_{e \in \edges} \sheafHom(\mass_x \otimes \sheaf_{t(e)}, \sheaf_{h(e)}) \oplus  \sheafHom(\mass_y \otimes \sheaf_{h(e)}, \sheaf_{t(e)}) \]
		We have a projection $\pi : \P^1 \times  \QM_\twist(\mathbb{P}^1, \XX_\eta, \degree) \to  \QM_\twist(\mathbb{P}^1, \XX_\eta, \degree)$. Then the deformations and obstructions are given by
		\[ \text{def} = R\pi^0(\P^1, \tgtquasimaps)  \ \ \text{obs} = R\pi^1(\P^1, \tgtquasimaps).\]

		\begin{lemma}
			The obstruction theory is symmetric.
		\end{lemma}
		\begin{proof}
			We must show $\text{def} = \text{obs}^{\vee}$. By Serre duality, we have $H^1(\P^1, \tgtquasimaps) = H^0(\P^1, \tgtquasimaps^{\vee} \otimes \omega_{\P^1})^{\vee}$. We have 
			\begin{align} \tgtquasimaps^{\vee} \otimes \omega_{\P^1} & = \bigoplus_{e \in \edges} \sheafHom(\sheaf_{h(e)}, \mass_x \otimes \sheaf_{t(e)}) \otimes \omega_{\P^1} \oplus  \sheafHom(\sheaf_{t(e)}, \mass_y \otimes \sheaf_{h(e)}) \otimes \omega_{\P^1} \\
				& = \bigoplus_{e \in \edges} \sheafHom(\mass_y \otimes \sheaf_{h(e)}, \sheaf_{t(e)}) \oplus  \sheafHom(\mass_x \otimes \sheaf_{t(e)}, \sheaf_{h(e)}) \\
				& = \tgtquasimaps \end{align} where we have used the condition $\mass_x \otimes \mass_y \cong \omega_{\P^1}^{-1}$. The result follows.
		\end{proof}
		
		Since the moduli space in our setting is smooth, its virtual degree equals its Euler characteristic up to a sign. We are thus led to the following definition.
		\begin{definition}
			\[ DT_\degree := \chi( \QM_\twist(\mathbb{P}^1, \XX_\eta, \degree)). \]
		\end{definition}
		We will also be interested in the natural refinement 
		\[ DT^{\refd}_\degree(\tau) := \sum_{i} (-\cohweight)^i \dim \H^i( \QM_\twist(\mathbb{P}^1, \XX_\eta, \degree), \C).\]
		
		We form generating functions for these quantities:  
		\begin{equation} \label{eq:DTgen} \DTgenerator(z) := \sum_{\degree} DT_\degree z^{\degree} \end{equation}
		and 	
		\begin{equation} \label{eq:redfDTgen}  \DTrefdgenerator(z, \tau) := \sum_{\degree} DT^{\refd}_\degree(\tau)z^{\degree}. \end{equation}
		
		The main result of this paper will give a surprising interpretation of these generating functions using symplectic duality. The next section establishes the groundwork for this result by introducing the `hypertoric loop space'.
		
		\begin{remark}
			Unlike the Euler characteristic $DT_\degree$, the polynomials $DT^{\refd}_\degree(\tau)$ might in principle be quite sensitive to deformations of the target. We make no claim to their invariance under such deformations, but believe they are nonetheless an interesting object of study, by analogy with other refined invariants arising from string theory.
		\end{remark}
		\section{Quasimaps to hypertoric varieties and the loop hypertoric space}

		\subsection{Heuristics and definition}
		In this subsection we describe our model for the (universal cover of) the loop space of a hypertoric variety, via finite dimensional approximations. 
		
		The basic idea is the following. The hypertoric variety $\XX_\eta$ was constructed as the symplectic reduction of $T^{\vee}\C^{\edges}$ by the torus $\gaugegrp$; one might naively expect that the symplectic reduction of the loop space $\cL T^{\vee}\C^{\edges}$ by the loop group $\cL \gaugegrp = \gaugegrp((t))$ would yield the loop space of $\XX_\eta$. Replacing $\gaugegrp((t))$ by $\gaugegrp[[t]]$ should define a covering space of the loop space with fibers $\gaugegrp((t)) / \gaugegrp[[t]] = \frak{g}_\Z \cong \H_2(\XX, \Z)$. Since $\XX_\eta$ is simply connected, this (again naively) should be the universal cover of the loop space. We perform a variant of this construction where $\gaugegrp[[t]]$ is replaced by its finite dimensional subgroup $\gaugegrp$, which is a natural choice from the perspective of quasimaps. 
		
		Since $\gaugegrp[[t]]/\gaugegrp$ is pro-unipotent, this is for many purposes a fairly mild difference. In particular, both $\tloops \XX_\eta$ and the space $\Maps(S^1, \XX_\eta)$ of continuous maps from the circle in $\XX_\eta$ carry an action of $S^1$ by `loop rotation', with isomorphic fixed-point loci given by infinitely many copies of $\XX_\eta$ indexed by the lattice $\H_2(\XX, \Z)$. It follows that the $S^1$- equivariant cohomology of $\tloops \XX_\eta$ and $\Maps(S^1, \XX_\eta)$ are isomorphic, perhaps after inverting the generator $u \in H^2_{S^1}(pt)$.

		We now describe the construction in detail. Recall that we defined $\XX$ starting from the sequence of tori $G \to D \to T$, where $D = (\C^{\times})^\coordset,$ together with the element $\eta \in \frak{g}_\Z^{\vee}$. 
		
		Let $\loops \C^\edges := \C^{\edges} \otimes_\C \C[t,t^{-1}]$. It is an infinite dimensional vector space with a basis $v_e \otimes t^k$, where $e \in \edges, k \in \Z$. It is filtered by the subspaces $\loops_\degreebound \C^{\edges}$, spanned by the basis elements with $|k| \leq \degreebound$. 
		
		Now consider the cotangent space $T^{\vee}\loops \C^{\edges} = T^{\vee}\C^{\edges} \otimes_\C \C[t,t^{-1}]$, with its induced basis $v_e \otimes t^k, w_e \otimes t^k$. It carries a symplectic form $\Omega$, given in terms of the symplectic form $\omega$ on $T^{\vee}\C^{\edges}$ by $$\Omega(v(t),w(t)) := \int_{|t|=1} \frac{dt}{t} \omega(v(t), w(t)).$$ In the coordinates $x_{e,k}, y_{e,k}$ determined by our chosen basis, we have $$\Omega = \sum_{k \in \Z} \sum_{e \in \coordset} d x_{e, k} \wedge d y_{e, -k}.$$
		
		The action of $D$ on $\C^{\edges}$ induces an action on $\loops \C^{\edges}$, preserving the filtration. Via the embedding $\gaugegrp \to D$, we have a hamiltonian action of $\gaugegrp$ on $T^{\vee} \loops_\degreebound \C^{\edges}$ with moment map 
		$$\mu_\degreebound(v(t),w(t)) := \int_{|t|=1} \mu(v(t),w(t)) \frac{dt}{t}. $$
		
		\[ \tloops_\degreebound \XX_\eta := \mu_\degreebound^{-1}(0) \sslash_\eta \gaugegrp. \]  
		We have natural closed embeddings $ \tloops_\degreebound \XX_\eta \to  \tloops_{\degreebound+1} \XX_\eta$, and we may define an ind-scheme by taking the limit along these embeddings.
		
		\begin{definition}
			\[ \tloops \XX_\eta = \lim_{\degreebound \to \infty} \tloops_\degreebound \XX_\eta. \]
		\end{definition} 
		
		The ind-scheme $\tloops \XX_\eta$ is a symplectic reduction of an infinite dimensional vector space by a finite dimensional torus. Were the vector space also finite dimensional, it would be a hypertoric variety in the usual sense. We will view $\tloops \XX_\eta$ as an ``ind- hypertoric variety.'' 
		
		Let $\loops D$ be the torus of automorphisms of $\loops \C^\edges$ sending each basis element to a multiple of itself.  We have $\loops D = (\C^\times)^{\loops D}$ where $\loops \edges := \edges \times \Z$. $\tloops \XX_{\eta}$ carries a Hamiltonian action of $\loops D$ which factors through $\loopspacetor := \loops D / G $.
		
		Morally, we may say that $\tloops \XX_\eta$ is the hypertoric variety associated to 
		\begin{enumerate}
			\item The set $\loops \edges := \edges \times \Z$.
			\item The short exact sequence of tori
			
			\begin{equation} \label{eq:looptori} \gaugegrp \to \loops D \to \loopspacetor. \end{equation}
			\item The character $\eta$.
		\end{enumerate}
		It is the limit of genuine hypertoric varieties associated to 
		\begin{enumerate}
			\item The set $\loops_\degreebound \edges := \edges \times [-\degreebound, \degreebound]$.
			\item The short exact sequence of tori
			
			\begin{equation} \label{eq:looptori2} \gaugegrp \to \loops_\degreebound D \to \loopspacetor_\degreebound. \end{equation}
			where $ \loops_\degreebound D = (\C^\times)^{\loops_\degreebound D}$.
			\item The character $\eta$.
		\end{enumerate} 
		Given an object associated to sequence \ref{eq:looptori}, there is often a natural `$\degreebound$-truncation' associated to the sequence \ref{eq:looptori2}. For instance, $\alpha \in \{ +, -\}^{\loops \edges}$ defines $\alpha \in \{ +, - \}^{\loops_\degreebound \edges}$ by restriction. We will speak of $\degreebound$-truncated items without further comment below; we hope that our meaning will be clear.
		
		There is a natural embedding $D \to \loops D$ (and thus $T \to \loopspacetor$) given by the `constant loops'. On the other hand, $\loopspacetor$ contains the `loop rotation' torus $\rotations$, acting by $(z^n \cdot x_{e,n}, z^{-n} \cdot y_{e,-n})$. We will for the most part ignore the infinite rank torus $\loopspacetor$, and focus on the subgroup 
		\begin{equation} \label{eq:finitetorusembedding} \rotations \times T \to \loopspacetor. \end{equation}
		
		\subsection{Lattice actions and fixed loci}
		We may identify $\fd_\Z$ with the group of diagonal matrices in $\End(\C^\edges)$ whose entries are powers of $t$. This defines a natural action of $\fd_\Z$ on $\loops \C^\edges$, which in turn induces a symplectic action on $T^{\vee} \loops \C^\edges$ and thus on $\tloops \XX_\eta$.  
		
		There is a natural embedding $\XX_\eta \to \tloops \XX_\eta$ given by the `constant loops'. The image is a connected component of the $\C^\times$-fixed locus. In fact, the connected components of the $\C^{\times}$-fixed locus are given by $\partial \gamma$-translates of $\XX_\eta$, where $\gamma \in \fg_\Z$. Recall that the latter may be identified with $H_2(\XX, \Z)$ via the Kirwan map.
		
		More explicitly, given $\chain \in \fd_\Z$, we can define $\chain \cdot T^{\vee}\C^{\edges}$ as the symplectic subspace of $T^{\vee} \loops \C^{\edges}$ given by the translation of the `constant loops' $T^{\vee}\C^{\edges}$ by $\chain$: 
		$$\chain \cdot T^{\vee}\C^{\edges} := \bigoplus_{e \in \edges} \left( v_e \otimes t^{\chain_e} \oplus w_e \otimes t^{-\chain_e} \right).$$
		\begin{lemma}
			The fixed locus of the loop rotation action of $\C^{\times}$ is given by $$\bigsqcup_{\gamma \in \fg_\Z} \partial \gamma \cdot \XX_\eta$$
			where $\partial \gamma \cdot \XX_\eta = \left( \partial \gamma \cdot T^{\vee}\C^{\edges} \right) \sslash_\eta \gaugegrp$.
		\end{lemma}
		\begin{lemma} \label{lem:loopfixedpoints}
			The fixed points of the $T \times \C^\times$ action are given by 
			\[ \bigsqcup_{\fixed \in \XX^T_{\eta}, \gamma \in \fg_\Z} \partial \gamma \cdot \fixed \]
		\end{lemma}
		
		We will be interested in the (lagrangian) attracting cells of these fixed points in $\tloops \XX_\eta$ with respect to certain $\C^{\times}$-actions. Namely, let $\chain \in \Z^{\edges}$ be a cocharacter of $D$ and $\loopnumber$ a positive integer, and consider the cocharacter $(\chain, \loopnumber)$ of $D \times \C^{\times}$. Roughly speaking, one may think of the attracting cell of $\partial \gamma \cdot p$ as parametrizing quasimaps from $\C$ to $\XX_\eta$, with `degree' $\gamma$ and limit $p$ at $z=0$.  We will eventually make this statement precise with Proposition \ref{prop:GITquotient}. 
		
		We will restrict ourselves to $(\chain, \loopnumber)$ for which $\loopnumber$ is much larger than $|\chain_e|$ for all $e \in \edges$. We thus treat $\chain$ as a small perturbation of the loop rotation cocharacter. This will make the combinatorics more intuitive, without fundamentally changing the nature of the results.

		\begin{definition} \label{def:loopcochar}
			Fix a pair $(\chain, \loopnumber)$ as above, and let $\loopcochar := (\chain, \loopnumber)$ be the corresponding cocharacter of $\rotations \times T$, and by abuse of notation, the induced cocharacter of $\rotations \times \loopspacetor$.
		\end{definition}

		\subsection{Combinatorial data}
		We can define as in Section \ref{sec:arrangements} a quantized hyperplane arrangement $\tloops_\degreebound \qarrangement_\eta$ associated to the sequence \ref{eq:looptori2} and the character $\eta$. It controls the module categories attached to the quantization of $\tloops_\degreebound \XX_\eta$. 
		
		We will also consider the arrangement $\tloops \qarrangement_\eta$ associated to Sequence \ref{eq:looptori}. Since this is an arrangement in the infinite dimensional space, special care is needed. We will in fact use $\tloops \qarrangement_\eta$ mainly as a convenient bookkeeping device for the combinatorics of the finite dimensional arrangements $\tloops_\degreebound \qarrangement_\eta$ as $\degreebound \to \infty$, and our results will always fundamentally concern finite arrangements and limits thereof. A better general framework in which to understand Sequence \ref{eq:looptori} is perhaps that of non-finitary matroids \cite{bruhn2013axioms}.

		Our first task is to describe the set $\tloops \bases$ of bases of $\tloops \arrangement_{\eta}$, by which we mean subsets of $\loops \edges$ whose $\degreebound$-truncations are bases of $\tloops_\degreebound \arrangement_{\eta}$ for all sufficiently large $\degreebound$. We expect from Lemma \ref{lem:loopfixedpoints} that they will be indexed by the bases of the finite arrangment and the elements of $\fg_\Z$. 
		
		Recall that each base $b \in \bases$ is a subset of $\edges$. Let $\tilde{b} := (\loops \edges \setminus \edges) \cup b$. 
		\begin{lemma}
			$\tloops \bases = \{ \partial \gamma \cdot \tilde{b} \text{ for } b \in \bases \text{ and } \gamma \in \fg_\Z \}$. 
		\end{lemma}
		Here we have used the action of the lattice $\fd_\Z$ on $\loops \edges$ by translation: $\chain \cdot e \times n = e \times (n + \chain_e)$. 
		
		\begin{proof}
			A subset $s \subset \loops \edges$ is a base if and only if its complement $s^c := \loops \edges \setminus s$ is a base for the dual sequence. By definition, $s^c$ is a base if the canonical map $\Z^{s^c} \to \fg^{\vee}$ is an isomorphism. We must then have $$s^c = \bigcup_{e \in a} e \times n_e$$ where $a \in \bases^!$ is a base of the dual (finite) arrangement. Let $b = a^c \in \bases$ and set $\gamma := \phi_b(n)$. Then $\partial \gamma \cdot \tilde{b} = s$. The converse is direct.  
		\end{proof}

		Next we identify the bounded feasible chambers of $\tloops \arrangement_\eta$. 
		
		\begin{definition}
			We say $\beta \in \{ +, - \}^{\loops \edges}$ is $\loopcochar$-bounded if its truncations $\beta \in \{ +, -\}^{\loops_\degreebound \edges}$ are $\loopcochar$-bounded for all $\degreebound$. We say it is $\eta$-feasible if it is feasible for all sufficiently large $\degreebound$.
			
			Write $\tloops \bounded^{\loopcochar}$ for the set of bounded chambers, and $\tloops \boufeas^{\loopcochar}_\eta$ for the set of bounded feasible chambers.
		\end{definition}
		Recall that $\alpha \in \{ +, - \}^{\edges}$ indexes a chamber of $\arrangement$. We define the following elements of $\{+, -\}^{\loops \edges}$.
		\begin{align*}
			\alpha^0 & :=  \{ - \}^{ \loops^{< 0} \edges} \times \prod_{e \in \edges} \{ \alpha(e) \}^{e \times 0} \times \{ + \}^{ \loops^{> 0} \edges} \\  
			\alpha^\infty & :=  \{ + \}^{ \loops^{< 0} \edges} \times \prod_{e \in \edges} \{ \alpha(e) \}^{e \times 0} \times \{ - \}^{ \loops^{> 0} \edges}.
		\end{align*}
		where $\loops^{> 0} \edges = \edges \times \Z^{>0}$ and $\loops^{< 0} \edges = \edges \times \Z^{< 0}$. 
		
		We have an action of $\fd_\Z$ on $\{ +, - \}^{\loops \edges}$ by translation. Namely, $\chain \cdot \beta (e \times n) := \beta(e \times n + \chain_e)$. Given $\chain \in \fd_\Z$, we have the translates
		\begin{align} \label{eq:loopchambers} (-\chain) \cdot \alpha^0 & = \{ - \}^{ \loops^{< \chain} \edges} \times \prod_{e \in \edges} \{ \alpha(e) \}^{e \times \chain_e} \times \{ + \}^{ \loops^{> \chain} \edges} \\
			(-\chain) \cdot \alpha^\infty & = \{ + \}^{ \loops^{< \chain} \edges} \times \prod_{e \in \edges} \{ \alpha(e) \}^{e \times \chain_e} \times \{ - \}^{ \loops^{> \chain} \edges}. 
		\end{align}
		where $\loops^{> \chain} \edges \subset \loops \edges$ is the subset $\bigcup_{e \in \edges} e \times (\chain_e, \infty)$ and likewise $\loops^{< \chain} \edges := \bigcup_{e \in \edges} e \times (-\infty, \chain_e)$. 
		
		\begin{lemma} \label{lem:boundedloopchambers}
			The set of bounded feasible chambers is given by
			$$\tloops \boufeas^{\loopcochar}_\eta = \{ \partial \gamma \cdot \alpha^0 \text{ for } \gamma \in \fg_\Z \text{ and } \alpha \in \boufeas^{\zeta}_\eta \}.$$
			Likewise,
			$$\tloops \boufeas^{-\loopcochar}_\eta = \{ \partial \gamma \cdot \alpha^\infty \text{ for } \gamma \in \fg_\Z \text{ and } \alpha \in \boufeas^{-\zeta}_\eta \}.$$
			
		\end{lemma}
		
		Write $$\boufeas^{\zeta}_\eta \xleftarrow{\mu_{\zeta}} \bases \xrightarrow{\mu_{-\zeta}} \boufeas^{-\zeta}_\eta$$ for the bijections defined as in Equation \ref{def:mu}. Write $$\tloops \boufeas^{\loopcochar}_\eta \xleftarrow{\mu_0} \tloops \bases \xrightarrow{\mu_\infty} \tloops \boufeas^{-\loopcochar}_\eta$$ for their loopy analogues.
		\begin{lemma}
			Let $b \in \bases$ be a base of the finite arrangement. Then $\mu_0(\partial \gamma \cdot \tilde{b}) = \partial \gamma \cdot \mu_\zeta(b)^0$ and $\mu_\infty(\partial \gamma \cdot \tilde{b}) = \partial \gamma \cdot \mu_{-\zeta}(b)^\infty$.
		\end{lemma}
		
		To these chambers we can associate lagrangians $\frak{L}(\partial \gamma \cdot\alpha^0), \frak{L}(\partial \gamma \cdot \alpha^{\infty}) \subset \tloops \XX_\eta$, via  Equation \ref{eq:lagfromchambers}. Note that we have changed our notation slightly so that the indexing chamber is no longer a subscript, to avoid cramped expressions. In the next section, we describe these lagrangians in more elementary terms and relate them to quasimaps.
		
		We can likewise define the truncated Lagrangians $\frak{L}_\degreebound(\partial \degree \cdot \alpha^0), \frak{L}_\degreebound(\partial \degree \cdot \alpha^\infty)  \subset \tloops_\degreebound \XX$. Their intersections stabilize in the following sense.
		\begin{lemma} \label{lem:stabilisingintersections}
			Let $\twist \in \twistspace$ and $\degreebound \gg 0$. Then 
			\[  \frak{L}_\degreebound(\twist \cdot \alpha_1^0) \cap \frak{L}_\degreebound(\partial \degree \cdot \alpha_{2}^\infty) = \frak{L}(\twist \cdot \alpha_1^0) \cap \frak{L}(\partial \degree \cdot \alpha_{2}^\infty) \] 
			
			where the left-hand side is viewed as a subvariety of $\tloops \XX$ via the natural embedding. \end{lemma}
		\begin{proof}
			This is direct from the definitions. 
		\end{proof}
		
		\subsection{Presenting the moduli of stable quasimaps as an intersection of lagrangians} \label{sec:presmoduli}

		In this section we relate our lagrangians to moduli of quasimaps. For the benefit of readers less familiar with the hyperplane arrangements considered above, we phrase these results in terms of explicit coordinates, before returning to our more hands-off approach in the following section.
		
		\begin{definition}
			Let $d$ be an integer. We define Lagrangian subspaces of $\loops T^{\vee}\C$ by
			\[ \Att_0(d) := \{ x_{e,k} =0 \text{ for } k < -d \text{ and }  y_{e,k} = 0 \text{ for } k \leq d  \} \]  
			\[ \Att_\infty(d) := \{ x_{e,k} =0 \text{ for } k \geq d  \text{ and }  y_{e,k} = 0 \text{ for } k > -d\} \]  
		\end{definition}
		
		We have 
		\begin{equation} \label{eq:atthom} \Att_0(d_1) \cap \Att_\infty(d_2) \cong \Hom(\mathcal{O}_{\P^1}(d_1), \mathcal{O}_{\P^1}(d_2 - 1)) \oplus \Hom(\mathcal{O}_{\P^1}(d_2), \mathcal{O}_{\P^1}(d_1 - 1) ). \end{equation}  
		
		\begin{definition}
			Given $\chain \in \fd_\Z$, we define a Lagrangian $\widehat{\Att}_0(\chain)$ in $\loops T^{\vee}\C^{\edges}$ by taking the product of factors $\widehat{\Att}_0(\chain_e)$ as above over all edges $e \in \edges$. We can similarly define $\widehat{\Att}_\infty(\chain)$.
		\end{definition}
		Let $\alpha_+ = \{ + \}^{\edges}$ and $\alpha_- = \{ - \}^{\edges}$. Then by construction we have
		\begin{lemma} \label{lem:lagsareatts}
			\[ \frak{L}(\chain \cdot \alpha_+^0) = \widehat{\Att}_0(\chain) \sslash_{\eta} \gaugegrp \] 
			and
			\[ \frak{L}(\chain \cdot \alpha_-^\infty) = \widehat{\Att}_\infty(\chain) \sslash_{\eta} \gaugegrp. \] 
		\end{lemma}
		
		We now fix a twist $\twist \in \twistspace$, as in Definition \ref{def:twistedquasimap}. Given $\degree \in \fg_\Z$, we have $\partial \degree + \twist \in \fd_\Z$.
		
		\begin{proposition} \label{prop:GITquotient}
			We have
			\begin{equation} \label{eq:moduliisintersec}  \QM_\twist(\mathbb{P}^1, \XX_\eta, \degree) = \frak{L}(\twist \cdot \alpha_+^0) \cap \frak{L}(\partial \degree \cdot \alpha_-^\infty) . \end{equation}
		\end{proposition}
		\begin{proof}
			The claim follows from Equation \ref{eq:atthom} and Lemma \ref{lem:lagsareatts}, together with Lemma \ref{lem:quasimapsareGITquotient} which presents the quasimap moduli space as the corresponding GIT quotient.
		\end{proof}
		
		\begin{corollary}
			For generic $\eta$,  $\QM_\twist(\mathbb{P}^1, \XX_\eta, \degree)$ is a smooth variety.
		\end{corollary}
		\begin{proof}
			By construction, the intersection on the right-hand of Equation \ref{eq:moduliisintersec} is a GIT quotient of the vector space $W := \widehat{\Att}_0(\twist)\cap \widehat{\Att}_\infty(\degree)$ by the torus $\gaugegrp$. If the stable and semistable loci coincide, the result is a toric orbifold. The orbifold structure corresponds to the existence of non-trivial (finite) stabilizers of stable orbits of $\gaugegrp$. By assumption, the embedding $\gaugegrp \to D$ defining our hypertoric variety $\XX$ is unimodular. It follows that the same is true for the action of $\gaugegrp$ on $W$. It follows that the GIT quotient is a smooth variety, as claimed.
		\end{proof}
		
		\subsection{From loop space lagrangians to loop space modules}
		
		From now on, we fix our twist $\twist$ to be the basepoint $\twist_0 \in \twistspace$ corresponding to $0 \in \fd_\Z$. This will ensures that the modules $\bigsimple_\degreebound(\alpha_+^0)^{\optionalchamber}$ and $\bigsimple_\degreebound(\partial \degree \cdot \alpha_-^\infty)^{\optionalchamber}$ belong to category $\cO$ for {\em opposite} choices of cocharacter. It is the choice for which our results have the cleanest form. The generalization to other twists, however, does not pose any essential difficulties.
		
		We thus take as our starting point the two lagrangians on the right-hand side of Equation \ref{eq:moduliisintersec}, with $\twist=\twist_0$. We can `quantize' these lagrangians as follows. Let $\tloops_\degreebound \cG_\eta$ be the category of Gelfand-Tsetlin modules associated to the arrangement $\tloops_\degreebound \arrangement_\eta$.
		\begin{definition} 
			Let $L_N(\alpha_+^0)$ (resp. $L_N(\partial \degree \cdot \alpha_-^\infty)$) be the simple objects of $\tloops_\degreebound \cG_\eta$ associated to the chambers $\alpha_+^0$ (resp. $\partial \gamma \cdot \alpha_-^\infty$) described in Equation \ref{eq:loopchambers}. 
		\end{definition}
		By Lemma \ref{lem:boundedloopchambers}, these modules lie in category $\tloops_\degreebound \cO_\eta^{\loopcochar}$ (resp. $\tloops_\degreebound \cO_\eta^{-\loopcochar}$) for each $\degreebound$.

		\begin{lemma} \label{lem:extsstabilize}
			The following Ext groups stabilize as $\degreebound \to \infty$: 
			\begin{align*} \lim_{\degreebound \to \infty} \Ext^{\bullet + \codim} \left( L_\degreebound(\alpha_+^0), L_\degreebound(\partial \degree \cdot \alpha_-^\infty) \right) \\ & = \H^{\bullet}( \frak{L}(\alpha_+^0)  \cap \frak{L}(\alpha_-^\infty), \C) \\
				& =  \H^{\bullet}(\QM_{\twist_0}(\mathbb{P}^1, \XX_\eta, \degree) , \C). \end{align*}
			Here $\codim = d_\degreebound(\alpha_+^0, \partial \gamma \cdot \alpha_-^\infty)$ is defined as in Proposition \ref{def:hamming}.  
		\end{lemma}
		
		\begin{proof}
			This follows from Proposition \ref{cor:extsarecohomology}, combined with Lemma \ref{lem:stabilisingintersections}.
		\end{proof}
		
		Note that the shift of grading by the codimension  diverges as $\degreebound \to \infty$.

		\section{The periodic hypertoric space} 
		
		We now turn to the symplectic dual to the loop space. We will apply the same combinatorial procedure that we would use for a finite type hypertoric variety to produce a candidate for the dual. It would be interesting to compare this with the more canonical approach of \cite{MR3952347}, via convolution algebras.
		
		To this end, we consider the Gale dual of the sequence \ref{eq:looptori}. It corresponds to the data of 
		\begin{enumerate}
			\item The set $\loops \edges$.
			\item The short exact sequence of tori 
			\begin{equation} \label{eq:pertori}
				\loopspacetor^{\vee} \to \loops D^{\vee} \to G^{\vee}.
			\end{equation}
			\item The character $-\loopcochar$ of $\loopspacetor^{\vee}$.
		\end{enumerate}
		We also consider the `truncated' data, namely:
		\begin{enumerate}
			\item The set $\loops^{\degreebound} \edges$.
			\item The short exact sequence of tori 
			\begin{equation} \label{eq:pertori2}
				\loopspacetor^{\vee}_\degreebound \to \loops_\degreebound D^{\vee} \to G^{\vee}.
			\end{equation}
			\item The restriction of the character $-\loopcochar$.
		\end{enumerate}
		We write $\per_\degreebound \arrangement^!_{-\loopcochar}$ for the associated hyperplane arrangement. If we let $\degreebound \to \infty$, we obtain a limiting hyperplane arrangement $\per \arrangement^!_{-\loopcochar}$. It is a hyperplane arrangement on $\frak{g}_\R$, given by all $\loopnumber \Z$-translates of the hyperplanes of $\arrangement^!_{-\zeta}$. It is preserved by the action of $\frak{g}_\Z$ by translations. 
		
		We can define by the usual prescription the associated hypertoric variety $\per_\degreebound \XX^!_{-\loopcochar}$. There is an {\em open} embedding $\per_\degreebound \XX^!_{-\loopcochar} \to \per_{\degreebound+1} \XX^!_{-\loopcochar}$ `dual' to the closed embedding $\tloops_\degreebound \XX \to \tloops_{\degreebound+1} \XX$. Thus we can take the limit of schemes
		\[ \per \XX^!_{-\loopcochar} := \lim_{\degreebound \to \infty} \per_\degreebound \XX^!_{-\loopcochar}. \]
		Morally, this is the hypertoric variety associated to Sequence \ref{eq:pertori}. When $\XX \cong T^{\vee}\P^1 \cong \XX^!$, $\per \XX^!$ is a symplectic surface containing an infinite chain of rational curves, whose hyperk\"ahler geometry has been studied in \cite{anderson1989complete}. The geometry in more general cases has been further explored in \cite{hattori2011volume, goto1992toric, dancer2017hypertoric, dancer2017hypertoric}. We learned of the space $ \per \XX^!_{-\loopcochar}$ many years ago from an unpublished note of Hausel and Proudfoot. 
		
		$ \per \XX^!_{-\loopcochar}$ carries an action of $\frak{g}_\Z$, which is free on an analytic open subset, which is also a homotopy retract. The quotient of this retract by $\frak{g}_\Z$ is called the hypertoric Dolbeault space in \cite{mcbreen2018homological}, and $ \per \XX^!_{-\loopcochar}$ plays the role of universal cover of the Dolbeault space. It is shown in \cite{dancso2019deletion} that when the hypertoric variety $\XX^!$ arises from a graph $\Gamma$, the quotient is closely related to the compactified Jacobian of a certain reducible nodal curve with dual graph $\Gamma$. In particular, they have the same cohomology.

		\begin{figure}[h]
			\centering{
				\resizebox{75mm}{!}{\includegraphics{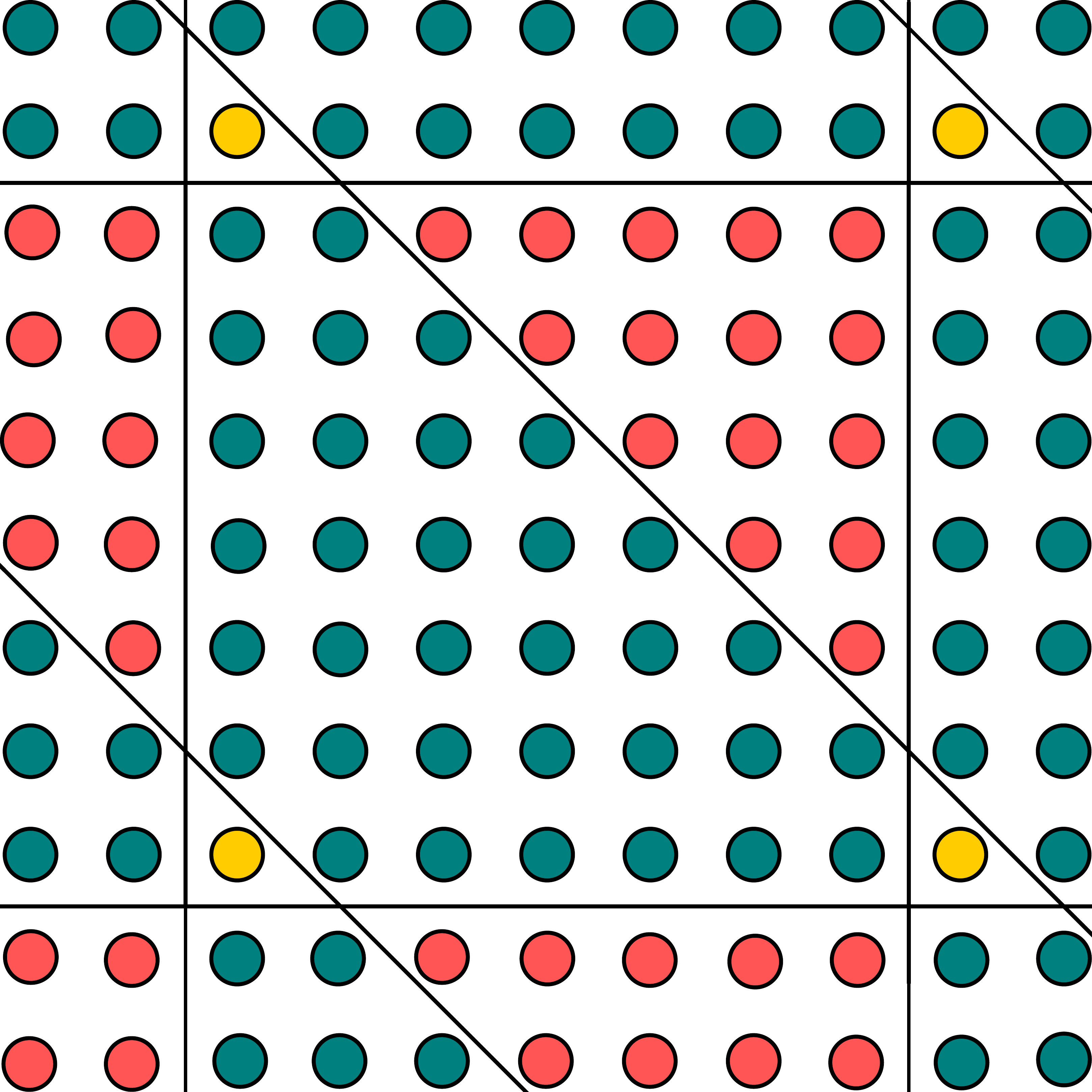}}
				\caption{The integer points of a periodic arrangement. Chambers related by the action of $\frak{g}_\Z \cong \Z^2$ have matching colors.}
				\label{fig:periodic}
			}
		\end{figure}
		Given $b \in \bases^!$, let $\tilde{b} \subset \loops \edges$ be the image of the composition of inclusions $b \subset \edges \subset \loops \edges$.
		\begin{lemma}
			The bases $\per \bases^!$ of $\per \arrangement^!_{-\loopcochar}$ are given by $\partial \gamma \cdot \tilde{b}$ for $b \in \bases^!$ and $\gamma \in \fg_\Z$. 
		\end{lemma}
		Write $\per \boufeas^{-\eta}_{-\loopcochar}$ for the $-\loopcochar$-feasible $-\eta$-bounded chambers in $\per  \arrangement^!_{-\loopcochar}$. We have the following analogue of Lemma \ref{lem:boundedloopchambers}.
		\begin{lemma} \label{lem:generatedbyfundamentalbox}
			$$\per \boufeas^{-\eta}_{-\loopcochar} = \{ \partial \gamma \cdot \alpha^0 \text{ for } \gamma \in \fg_\Z \text{ and } \alpha \in \boufeas^{-\eta}_{-\zeta} \}.$$
			
		\end{lemma}
		This lemma provides a bijection between the irreducible lagrangian components of the core of $\per \XX^!$ and the $\fg_\Z$-translates of the $-\eta$-bounded components of $\XX^!$.
		
		\subsection{Enumerative invariants as traces}
		
		\begin{lemma}
			For $\degreebound \gg 0$, we have an isomorphism
			\begin{equation} \label{eq:almosthere} e_{\partial \degree \cdot \alpha_-^\infty} \left( \per_\degreebound (R^!)^{-\loopcochar} \right) e_{\alpha_+^0} \cong  \H^{\bullet - \codim}(\QM_{\twist_0}(\mathbb{P}^1, \XX_\eta, \degree), \C). \end{equation} 
		\end{lemma}
		Here $\codim = d_\degreebound(\alpha_+^0, \partial \degree \cdot \alpha^\infty_-)$ denotes the codimension of $\frak{L}_{\degreebound, \partial \degree \cdot \alpha_-^\infty} \cap \frak{L}_{\degreebound, \alpha_+^0}$ in either lagrangian. It is an instance of Definition \ref{def:hamming}.
		
		\begin{proof}
			This is a consequence of Lemma \ref{lem:extsstabilize}, which identifies the right-hand side with an Ext group in $\cG_\eta$, together with Corollary \ref{cor:extsonthedual}, which identifies this Ext group with the left-hand side via Koszul duality. 
		\end{proof}

		The left-hand side of Equation \ref{eq:almosthere} equals a certain weight space in an indecomposable tilting module. Namely, we have bijections
		\[ \per (\boufeas^!)_{-\loopcochar}^{-\eta} \xleftarrow{\mu_0} \per \bases^! \xrightarrow{\mu_\infty} \per (\boufeas^!)_{\loopcochar}^{-\eta} \]
		between vertices and bounded feasible chambers for $\pm \loopcochar$, as in Diagram \ref{eq:flipbijections}. The composition of these bijections from left to write defines a bijection $\nu$ of chambers. By Proposition  \ref{prop:crossedweightsandtiltings} we have a graded isomorphism
		\begin{equation} \label{eq:gradingshift} e_{\partial \degree \cdot \alpha_-^\infty} \left( \per_\degreebound (R^!)^{-\loopcochar} \right) e_{\alpha_+^0} \langle -d_\degreebound(\alpha_+^0, \nu(\alpha_+^0)) \rangle = e_{\partial \degree \cdot \alpha_-^\infty}\tilting_{\degreebound, \nu(\alpha_+^0)}. \end{equation}
		
		The chamber $\nu(\alpha_+^0)$ which appears in this expression can be described explicitly. Let $b = \mu^{-1}(\alpha_+)$. Then
		\[ \nu(\alpha_+^0) = \nu(\alpha_+)^\infty = \{ + \}^{ \loops^{< 0} \edges} \times \prod_{e \in b} \{ + \}^{e \times 0} \times \prod_{e \notin b} \{ - \}^{e \times 0} \times \{ - \}^{ \loops^{> 0} \edges}. \]
		In particular, $\nu(\alpha_+^0)$ differs from $\alpha_-^{\infty}$ in precisely $|b| = \rk T$ places, where $T = D / \gaugegrp$.
		
		Comparing the grading shifts in Equations \ref{eq:gradingshift} and \ref{eq:almosthere}, we find a graded isomorphism for $\degreebound \gg 0$
		\[ e_{\partial \degree \cdot \alpha_-^\infty}\tilting_{\degreebound, \nu(\alpha_+^0)} \cong \H^{\bullet - d_{\degreebound, \gamma}}(\QM_{\twist_0}(\mathbb{P}^1, \XX_\eta, \degree), \C) \]
		where \begin{equation} \label{eq:truncateddegreedifference} d_{\degreebound, \degree} :=  d_\degreebound(\alpha_+^0, \partial \degree \cdot \alpha^\infty_-) -d_\degreebound(\alpha_+^0, \nu(\alpha_+^0)).\end{equation} Although both terms in this degree shift diverge as $\degreebound \to \infty$, their difference equals $-d_N(\nu(\alpha^0_+),\partial \gamma \cdot \alpha_-^{\infty})$ which converges to 
		\[ d_\gamma := -|\partial \degree | + \rk T. \]
		where $|\chain| = \sum_{e \in \edges} |\chain_e|$.
		
		We can thus derive the following expression for the generating function defined in Equation \ref{eq:redfDTgen}. 
		
		\begin{theorem} \label{thm:mainthm}
			\begin{equation} \label{eq:mainthm} \DTrefdgenerator(z, \tau) = \sum_{\degree \in \degreelattice} \grdim \left(  e_{\partial \degree \cdot \alpha_-^\infty}\tilting_{\nu(\alpha_+^0)} \right) z^\degree \tau^{-d_{\degree}}. \end{equation}
		\end{theorem}
		Loosely speaking, the right hand side is a graded trace of an indecomposable tilting module over $\per \XX^!_{-\loopcochar}$.

		\subsection{Verma filtrations and explicit formulae}
		Theorem \ref{thm:mainthm} has the benefit of being stated in fairly general terms - one can imagine a similar statement holding for non-hypertoric symplectic resolutions. Moreover, its proof does not require us to know either side explicitly.
		
		Nevertheless, we can deduce from Theorem \ref{thm:mainthm} a more explicit formula for the left-hand side, using the filtration of $\tilting_{\nu(\alpha_+)^\infty}$ from Proposition \ref{prop:tiltingfiltration}. This requires us to plunge back into the combinatorics of our hyperplane arrangements. The end result can also be obtained by a direct analysis of the quasimap spaces, but we find the treatment via symplectic duality both instructive and suggestive of possible generalizations.
		
		In order to apply the proposition, our first task is to understand for which $c \in \per \bases^!$ does $\per \bounded^{-\eta}_c$ contain the chamber $\alpha_+^0$. Recall that $\per \bases^! = \{ \partial \gamma \cdot \tilde{b} \}$ for $b \in \bases^!$, $\gamma \in \fg_\Z$. It will be helpful to parametrize $\gamma$ using the isomorphism $\phi_b : \Z^{b} \to \fg_\Z$. 
		
		\begin{definition} \label{def:posdegsb}
			Let $b \in \bases^!$. Write $\mu : \B^! \to \boufeas^{-\eta}_{-\zeta}$. Let $\G^b \subset \Z^{b}$ be the subgroup generated by 
			\begin{align}
				(\Z^{\leq 0})^e & \text{ for } \mu(b)(e) = + \\
				(\Z^{\geq 0})^e & \text{ for }  \mu(b)(e) = - 
			\end{align}
		\end{definition}
		It is the set of $s \in \Z^b$ for which $\langle \eta, \phi_b(s) \cdot x \rangle > \langle \eta, x \rangle$. 
		\begin{definition} \label{def:smon}
			Let $\smon_\alpha^b \subset \G^{b}$ be the submonoid generated by 
			\begin{align}
				(\Z^{\leq 0})^e & \text{ for } \alpha(e)= \mu(b)(e) = + \\
				(\Z^{\geq 0})^e & \text{ for } \alpha(e) =  \mu(b)(e) = - \\
				(\Z^{< 0})^e & \text{ for } \alpha(e) \neq  \mu(b)(e) = + \\ 
				(\Z^{> 0})^e & \text{ for } \alpha(e) \neq \mu(b)(e) = -  
			\end{align}
		\end{definition}
		
		\begin{lemma} \label{lem:periodiccone}
			Let $b \in \bases$ and $s \in \Z^{b}$. Let $\alpha \in \{ +, - \}^{\edges}$. Then $\alpha^0 \in \per \bounded^{-\eta}_{\phi_b(s) \cdot \tilde{b}}$ if and only if $s \in \smon_{\alpha}^b$. 
		\end{lemma}
		\begin{proof}
			Translating by $\phi_b(-s)$, we find that the desired inclusion holds if and only if
			$$\mu_0(\tilde{b})(e) = \phi_b(-s) \cdot \alpha^0(e)$$ 
			for all $e \in \tilde{b}$. We have $\mu_0(\tilde{b}) = \mu(b)^0$ and $\mu(b)^0(e) = \mu(b)(e)$ for $e \in b$. The condition thus becomes 
			$$ \mu(b)(e) = \phi_b(-s) \cdot \alpha^0(e)$$
			for all $e \in b$. One can then check directly that this holds only for $s$ as described.
		\end{proof}
		
		Combining the above with Proposition \ref{prop:tiltingfiltration}, we conclude the following. 
		\begin{proposition} \label{prop:filtrationper} 
			There is a filtration of $\tilting_{\nu(\alpha_+)^\infty}$ indexed by $c \in \per \bases^!$, whose nonzero subquotients are given by $V^!_{\mu_{\infty}(c)}$ for $c = \phi_b(s) \cdot \tilde{b}$ where $b \in \bases^!$ and $s \in \smon^b_{\alpha_+}$.
		\end{proposition}
		
		\begin{lemma}
			The weight space
			\[ e_{\phi_b(k) \cdot \alpha_-^\infty} V^!_{\mu_\infty(\phi_b(s) \cdot \tilde{b})} \]
			equals $\C$ when $k-s \in \smon^b_{\alpha_-}$ and vanishes otherwise.
		\end{lemma}
		\begin{proof}
			This is an application of Lemma \ref{lem:vermaweightspaces} to our setting. Translating by $\phi_b(-k)$, we reduce to the case
			\[ e_{\alpha_-^\infty} V^!_{\mu_\infty(\phi_b(s-k) \cdot \tilde{b})}. \]  
			This is nonzero exactly when $\alpha_-^\infty \in \per \bounded^{-\eta}_{\phi_b(s-k) \cdot \tilde{b}}$. By a variation on Lemma \ref{lem:periodiccone}, this holds when $k - s \in \smon^{b}_{\alpha_-}$. 
		\end{proof}
		
		The contribution of the subquotient $V^!_{\mu_{\infty}(\phi_b(s) \cdot \tilde{b})}$ to our generating function $\DTgenerator(z)$ is thus $$\sum_{k | k - s \in \smon^{b}_{\alpha_-}} z^{\phi_b(k)} = \sum_{r \in \smon^{b}_{\alpha_-}} z^{\phi_b(s + r)}.$$ To obtain the contribution to the refined generating function $\DTrefdgenerator(z, \tau)$, we must take into account the $\Z$-grading on the subquotient. 
		
		Given $\chain \in \fd_\Z$, let $|\chain| := \sum_{e \in \edges} |\chain_e|$. Given $b \in \bases^!$, define $\epsilon_b \in \fd_\Z$ by $\epsilon_b^e = 1$ for $\mu(b)(e) = +$ and $\epsilon_b^e = 0$ for $\mu(b)(e) = -$. Thus \[|\epsilon_b| = d(\alpha_-, \mu(b)) = d(\alpha_-^{\infty}, \mu(b)^{\infty}) = d(\alpha_-^{\infty}, \mu_\infty(\tilde{b})). \] 
		
		%By construction, the top weight space of \[ V^!_{\mu_\infty(\phi_b(s) \cdot \tilde{b})} \] is supported in cohomological degree zero. The weight space 
		%\[ e_{\phi_b(s) \cdot \tilde{b})}V^!_{\mu_\infty(\phi_b(s) \cdot \tilde{b})} \] 
		%is 
		%$d(\nu(\alpha_+)^\infty, \mu_\infty(\phi_b(s) \cdot \tilde{b})) + d(\mu_\infty(\phi_b(s) \cdot \tilde{b}), \phi_b(k)\cdot \alpha_-^{\infty})$. The second term equals $d(\alpha_-^{\infty}, \mu_\infty(\phi_b(s-k) \tilde{b}))$
		
		\begin{lemma}
			The weight space 
			\[ e_{\phi_b(k) \cdot \alpha^\infty_-}V^!_{\mu_\infty(\phi_b(s) \cdot \tilde{b})} \] 
			is supported in cohomological degree 
			\[ \psi_{b}(k,s) := |\phi_b(k)| + |\phi_b(k - s) - \epsilon_b| - |\phi_b(s) + \epsilon_b|. \]
		\end{lemma}
		\begin{proof}
			By Theorem \ref{thm:mainthm}, the cohomological degree is given by the natural $\Z$-grading on the module $T^!_{\nu(\alpha_+^0)}$, shifted by $-d_{\phi_b(k)}$. This grading may be computed on any sufficiently large truncation of the periodic arrangement. By Lemma \ref{lem:vermaweightspaces}, this equals
			\[ d_N(\alpha_+^0, \mu_\infty(\phi_b(s) \cdot \tilde{b})) + d_N(\mu_\infty(\phi_b(s) \cdot \tilde{b}), \phi_b(k) \cdot \alpha^\infty_-) - d_N(\alpha_+^0, \nu(\alpha_+^0)) - d_{\phi_b(h)}\] for $N \gg 0.$ 
			We can rewrite this as
			\begin{align*} & d_N(\alpha_+^0, \mu_\infty(\phi_b(s) \cdot \tilde{b})) + d_N(\mu_\infty(\phi_b(s) \cdot \tilde{b}), \phi_b(k) \cdot \alpha^\infty_-) \\ & - d_\degreebound(\alpha_+^0, \phi_b(k) \cdot \alpha^\infty_-) + d_{N,\phi_b(k)} - d_{N,\phi_b(k)} \end{align*} for $N \gg 0.$
			
			The difference of the first and third terms gives $|\phi(k)| - |\phi_b(s) + \epsilon_b|$, and the second term equals $|\phi_b(k-s) - \epsilon_b|$. 
			
			%We can understand this by first supposing that $\mu_\infty(\tilde{b}) = \alpha_-^{\infty}$. Then the difference of the first and third terms equals $|\phi_b(k)| - |\phi_b(s)|$, whereas the middle term gives $|\phi_b(k-s)|$.

		\end{proof}
		
		Adding the contributions of each base $b \in \bases^!$, we finally obtain
		\begin{theorem} \label{thm:explicit} 
			\begin{equation} \DTrefdgenerator(z, \tau) = \sum_{b \in \bases^!} \sum_{s \in \smon^b_{\alpha_+}, r \in \smon^{b}_{\alpha_-}} \tau^{\psi_b(r+s,s)} z^{\phi_b(s + r)}. \end{equation}
		\end{theorem}
		\begin{example}
			We consider one of the simplest non-trivial examples, for which $\XX \cong T^{\vee}\P^2$ and $\XX^!$ is a resolution of the singularity $xy = z^3$. Both $\XX$ and $\XX^!$ are cographical, and in this case Gale duality is an instance of planar graph duality. 
			
			Thus, let $\Gamma$ be the graph with two vertices $v_1, v_2$ and three edges $e_1, e_2, e_3$ from $v_1$ to $v_2$. We pick the basis $(1,-1)$ of $C^0(\Gamma, \Z) / \Z(1,1)$ and $(0,1,0), (0,0,1)$ of  $H^1(\Gamma, \Z)$. 
			The associated sequence of tori is thereby identified with
			\[ \gaugegrp \cong \C^\times \to (\C^\times)^{\edges} \to (\C^\times)^2 \cong T. \]
			We pick the character $\eta = 1$ of $\gaugegrp$ and the cocharacter $\zeta = (-1,1)$ of $T$. Then $\XX(\Gamma)_\eta \cong T^{\vee} \P^2$.
			
			The dual graph $\Gamma^!$ is given by the cycle \[ \xrightarrow{e_1} w_{12} \xrightarrow{e_2} w_{23} \xrightarrow{e_3} w_{31} \xrightarrow{e_1}.\] We have $\XX(\Gamma^!)_{-\zeta} \cong \widetilde{\C^2/ \Z_3}$. The bases $b \in \bases^!$ are given by single edges $b_i = e_i$, $i =1,2,3$. We have 
			\begin{align*} 
				\mu(b_1) & = \{ +, - , +\} \\
				\mu(b_2) & = \{ +, +, + \} \\
				\mu(b_3) & = \{ -, -, + \}. 
			\end{align*} 
			The maps $\phi_b$ are identified in our bases with the identity $\Z \to \Z$, and the monoids $\smon^b_{\alpha_+}$ are all equal to $\Z^{\geq 0}$. We find
			\[ \psi_{b_1}(k,s) = 6(k-s) - 4,  \psi_{b_2}(k,s) = 6(k-s) - 6, \psi_{b_3}(k,s) = 6(k-s) - 2. \]
			
			We conclude
			\begin{align} \DTrefdgenerator(z, \tau) & = \sum_{s \in \Z^{\geq 0}, r \in \Z^{> 0}} (\tau^{6r-2} + \tau^{6r-4} + \tau^{6r-6} )z^{(s + r)}. \end{align}
		\end{example}
		
		\begin{example} \label{ex:flags}
			Consider the linear quiver $Q$ with vertices $v_1, ..., v_N$ and arrows $v_i \to v_{i+1}$. Representations of this quiver in the category of coherent sheaves on a curve $C$, which assign a locally free sheaf $\cV_i$ of rank $r_i$ to each vertex and maps of sheaves $\cV_i \to \cV_{i+1}$ for each edge, are an interesting object of study in enumerative geometry.
			
			Let $Q^{\operatorname{ab}}$ be the abelianization of $Q$. Thus, fix a tuple of integers $r_1, ..., r_N$, and define the {\em abelianized} quiver $Q^{\operatorname{ab}}$ to have vertices $v^{j}_{i}, j =1, ..., r_i$ and edges $v^{j}_i \to v^{j'}_{i+1}$ for all $j,j'$. We describe the twisted quasimap invariants of the variety $\XX(Q^{\operatorname{ab}})$.

			\begin{figure}[h]
				
				\begin{center}
					
					\tikzset{every picture/.style={line width=0.75pt}} %set default line width to 0.75pt        
					
					\begin{tikzpicture}[x=0.75pt,y=0.75pt,yscale=-1,xscale=1]
						
						%Shape: Circle [id:dp8022941951862667] 
						\draw   (120.75,160.88) .. controls (120.75,155.7) and (124.95,151.5) .. (130.13,151.5) .. controls (135.3,151.5) and (139.5,155.7) .. (139.5,160.88) .. controls (139.5,166.05) and (135.3,170.25) .. (130.13,170.25) .. controls (124.95,170.25) and (120.75,166.05) .. (120.75,160.88) -- cycle ;
						%Shape: Circle [id:dp1048950403174822] 
						\draw   (231.25,220.13) .. controls (231.25,215.09) and (235.34,211) .. (240.38,211) .. controls (245.41,211) and (249.5,215.09) .. (249.5,220.13) .. controls (249.5,225.16) and (245.41,229.25) .. (240.38,229.25) .. controls (235.34,229.25) and (231.25,225.16) .. (231.25,220.13) -- cycle ;
						%Shape: Circle [id:dp03436087043630065] 
						\draw   (370.75,40.38) .. controls (370.75,35.2) and (374.95,31) .. (380.13,31) .. controls (385.3,31) and (389.5,35.2) .. (389.5,40.38) .. controls (389.5,45.55) and (385.3,49.75) .. (380.13,49.75) .. controls (374.95,49.75) and (370.75,45.55) .. (370.75,40.38) -- cycle ;
						%Shape: Circle [id:dp06531347358520134] 
						\draw   (231.25,90.13) .. controls (231.25,85.09) and (235.34,81) .. (240.38,81) .. controls (245.41,81) and (249.5,85.09) .. (249.5,90.13) .. controls (249.5,95.16) and (245.41,99.25) .. (240.38,99.25) .. controls (235.34,99.25) and (231.25,95.16) .. (231.25,90.13) -- cycle ;
						%Shape: Circle [id:dp6488310129449609] 
						\draw   (370.75,159.88) .. controls (370.75,154.7) and (374.95,150.5) .. (380.13,150.5) .. controls (385.3,150.5) and (389.5,154.7) .. (389.5,159.88) .. controls (389.5,165.05) and (385.3,169.25) .. (380.13,169.25) .. controls (374.95,169.25) and (370.75,165.05) .. (370.75,159.88) -- cycle ;
						%Shape: Circle [id:dp24510371190906466] 
						\draw   (371.25,269.88) .. controls (371.25,264.7) and (375.45,260.5) .. (380.63,260.5) .. controls (385.8,260.5) and (390,264.7) .. (390,269.88) .. controls (390,275.05) and (385.8,279.25) .. (380.63,279.25) .. controls (375.45,279.25) and (371.25,275.05) .. (371.25,269.88) -- cycle ;
						%Straight Lines [id:da9704563418666788] 
						\draw    (139.5,160.88) -- (229.67,91.35) ;
						\draw [shift={(231.25,90.13)}, rotate = 502.36] [color={rgb, 255:red, 0; green, 0; blue, 0 }  ][line width=0.75]    (10.93,-3.29) .. controls (6.95,-1.4) and (3.31,-0.3) .. (0,0) .. controls (3.31,0.3) and (6.95,1.4) .. (10.93,3.29)   ;
						%Straight Lines [id:da9412690882684439] 
						\draw    (139.5,160.88) -- (229.57,219.04) ;
						\draw [shift={(231.25,220.13)}, rotate = 212.85] [color={rgb, 255:red, 0; green, 0; blue, 0 }  ][line width=0.75]    (10.93,-3.29) .. controls (6.95,-1.4) and (3.31,-0.3) .. (0,0) .. controls (3.31,0.3) and (6.95,1.4) .. (10.93,3.29)   ;
						%Straight Lines [id:da10318987146524583] 
						\draw    (249.5,90.13) -- (368.9,41.13) ;
						\draw [shift={(370.75,40.38)}, rotate = 517.69] [color={rgb, 255:red, 0; green, 0; blue, 0 }  ][line width=0.75]    (10.93,-3.29) .. controls (6.95,-1.4) and (3.31,-0.3) .. (0,0) .. controls (3.31,0.3) and (6.95,1.4) .. (10.93,3.29)   ;
						%Straight Lines [id:da703199686862915] 
						\draw    (249.5,90.13) -- (369.02,158.88) ;
						\draw [shift={(370.75,159.88)}, rotate = 209.91] [color={rgb, 255:red, 0; green, 0; blue, 0 }  ][line width=0.75]    (10.93,-3.29) .. controls (6.95,-1.4) and (3.31,-0.3) .. (0,0) .. controls (3.31,0.3) and (6.95,1.4) .. (10.93,3.29)   ;
						%Straight Lines [id:da9405289535432632] 
						\draw    (249.5,90.13) -- (370.13,268.22) ;
						\draw [shift={(371.25,269.88)}, rotate = 235.89] [color={rgb, 255:red, 0; green, 0; blue, 0 }  ][line width=0.75]    (10.93,-3.29) .. controls (6.95,-1.4) and (3.31,-0.3) .. (0,0) .. controls (3.31,0.3) and (6.95,1.4) .. (10.93,3.29)   ;
						%Straight Lines [id:da49930493116117947] 
						\draw    (249.5,220.13) -- (369.4,269.12) ;
						\draw [shift={(371.25,269.88)}, rotate = 202.23] [color={rgb, 255:red, 0; green, 0; blue, 0 }  ][line width=0.75]    (10.93,-3.29) .. controls (6.95,-1.4) and (3.31,-0.3) .. (0,0) .. controls (3.31,0.3) and (6.95,1.4) .. (10.93,3.29)   ;
						%Straight Lines [id:da421933391464557] 
						\draw    (249.5,220.13) -- (368.96,160.76) ;
						\draw [shift={(370.75,159.88)}, rotate = 513.5799999999999] [color={rgb, 255:red, 0; green, 0; blue, 0 }  ][line width=0.75]    (10.93,-3.29) .. controls (6.95,-1.4) and (3.31,-0.3) .. (0,0) .. controls (3.31,0.3) and (6.95,1.4) .. (10.93,3.29)   ;
						%Straight Lines [id:da8527651720394891] 
						\draw    (249.5,220.13) -- (369.63,42.03) ;
						\draw [shift={(370.75,40.38)}, rotate = 484] [color={rgb, 255:red, 0; green, 0; blue, 0 }  ][line width=0.75]    (10.93,-3.29) .. controls (6.95,-1.4) and (3.31,-0.3) .. (0,0) .. controls (3.31,0.3) and (6.95,1.4) .. (10.93,3.29)   ;
						%Straight Lines [id:da6023143017228322] 
						\draw  [dash pattern={on 4.5pt off 4.5pt}]  (130.33,3.33) -- (130.33,282.67) ;
						%Straight Lines [id:da9576442730470455] 
						\draw  [dash pattern={on 4.5pt off 4.5pt}]  (240.33,4) -- (240.33,285.33) ;
						%Straight Lines [id:da7607771903489898] 
						\draw  [dash pattern={on 4.5pt off 4.5pt}]  (380,0.75) -- (380.33,284) ;
						
						% Text Node
						\draw (120,127) node [anchor=north west][inner sep=0.75pt]    {$\cV^{1}_{1}$};
						% Text Node
						\draw (234.5,54.5) node [anchor=north west][inner sep=0.75pt]    {$\cV^{2}_{2}$};
						% Text Node
						\draw (234,189) node [anchor=north west][inner sep=0.75pt]    {$\cV^{1}_{2}$};
						% Text Node
						\draw (374.5,4.5) node [anchor=north west][inner sep=0.75pt]    {$\cV^{3}_{3}$};
						% Text Node
						\draw (374.5,124) node [anchor=north west][inner sep=0.75pt]    {$\cV^{2}_{3}$};
						% Text Node
						\draw (375,234) node [anchor=north west][inner sep=0.75pt]    {$\cV^{1}_{3}$};
						% Text Node
						\draw (125,290.17) node [anchor=north west][inner sep=0.75pt]    {$r_{1}$};
						% Text Node
						\draw (233.83,290.17) node [anchor=north west][inner sep=0.75pt]    {$r_{2}$};
						% Text Node
						\draw (374.5,290.17) node [anchor=north west][inner sep=0.75pt]    {$r_{3}$};
					\end{tikzpicture}
					
					\caption{The quiver $Q^{\operatorname{ab}}$ from Example \ref{ex:flags}, where $Q$ is the linear quiver with three vertices and ranks $r_1 = 1, r_2 = 2, r_3 = 3$. The dotted lines join vertices which were `split off' from a single vertex $v_i$ of the original quiver $Q$, of rank $r_i$.}
				\end{center}	
				
			\end{figure}
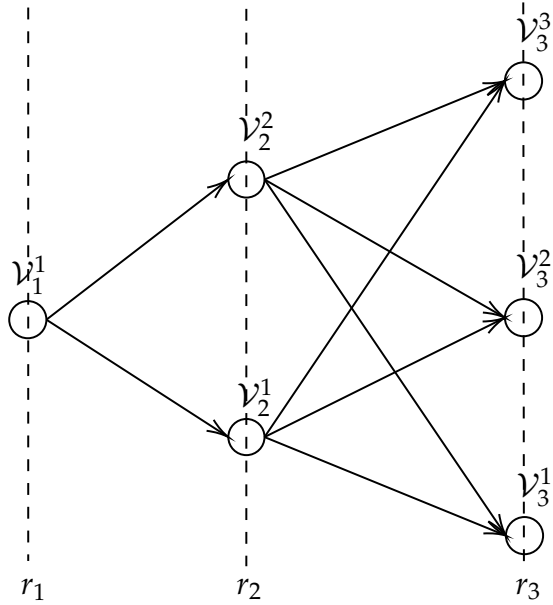
			
			Fix a sufficiently generic cocharacter $\zeta \in H^1(\Gamma, \Z)$. By Lemma \ref{lem:graphindexing}, the set $\bases$ is given by all spanning trees of $Q^{\operatorname{ab}}$. Given $b \in \bases$, its contribution to the sum in Theorem \ref{thm:explicit} is determined by $\phi_b$ (also described in Lemma \ref{lem:graphindexing}) and the monoids $\smon^b_{\alpha_+}$ and $\smon^{b}_{\alpha_-}$.
			
			In turn, these can be written down directly from Definition \ref{def:smon} once we know $\mu(b)$ and $\alpha_+, \alpha_-$. As usual, we have $\alpha_+ = \{ + \}^\edges$. On the other hand, $\alpha_-$ depends on the choice of $\zeta$; by making a suitable choice, we may ensure that $\alpha_-(e) = +$ for $e : v_i^{j} \to v_{i+1}^{j'}$ if and only if $j = r_i$.
			
			Finally, we describe $\mu(b)$. Recall that the complement $b^c$ of $b$ is a spanning tree. Let $e \in b$; then $H_1(e \cup b^c, \Z) \cong \Z$. Choose a generator $L^b_e$ which crosses $e$ in the positive direction, which we view as a loop in $\Gamma$. Then $\mu(b)$ is given by $\{ + \}^{b^c} \times \{ \sign \zeta(L^b_e) \}^{e \in b}$.
			
			Combined with Theorem \ref{thm:explicit}, this describes $\DTrefdgenerator(z,\tau)$ for $\XX(Q^{\operatorname{ab}})$. Since the number of spanning trees is quite large even for small numbers of vertices, we do not write the sum out in full.
		\end{example}

\noindent{[1] Chinese University of Hong Kong, Sha Tin, N.T., Hong Kong,  China} \newline
\noindent{[2] Massachusetts Institute of Technology (MIT), IAiFi Institute, 182 Memorial Drive, Cambridge, MA 02139, USA}  \newline
\noindent{[3] Yanqi Lake Beijing Institute for Mathematical Sciences and Applications (BIMSA). Huairou, Beijing, China}\newline
\noindent{[4] National Research University Higher School of Economics, Russian Federation, Lab- oratory of Mirror Symmetry, NRU HSE, 6 Usacheva str., Moscow, Russia} \newline
\noindent{[5] Yau center of mathematics, Tsinghua university,  Beijing, China}\newline \newline
\noindent{\tt{mcb@math.cuhk.edu.hk, artan@mit.edu, styau@tsinghua.edu.cn}}

\bibliographystyle{amsalpha}

%  This inserts the bib file
\bibliography{minimalbib}
\end{document}